\documentclass{amsart}

\usepackage{lmodern}
\usepackage{amsmath, amsthm, amsfonts, amssymb, amscd, hyperref, marginnote, todonotes, accents, color, multicol, mathtools}
\usepackage{amsxtra}     % Use various AMS packages
\usepackage[margin=1in]{geometry}
\usepackage{latexsym,xypic,hyperref,rotating}
\usepackage{tikz}
\usetikzlibrary{patterns}
\usetikzlibrary{matrix,arrows}
\usepackage{lscape} %landscape mode  % \begin{landscape}... \end{landscape
\usepackage{epsfig}
\usepackage{verbatim}
\usepackage{arydshln}
\usepackage[all,cmtip]{xy}
\usepackage{mathrsfs}
\usepackage{enumerate}
\allowdisplaybreaks
\usepackage[utf8]{inputenc}
\usepackage{cleveref}
\usepackage{colonequals}
\usepackage[all,cmtip]{xy}
\usepackage{tikz-cd}
\newcommand{\clC}[1]{\mathbf{C}(#1)}

\newcommand{\clT}{\mathbf{T}}

\newcommand{\clB}{\mathbf{B}}

\newcommand{\clG}[1]{\mathbf{G}(#1)}

\newcommand{\clH}[1]{\mathbf{H}(#1)}

\newcommand{\bm}{\mathbf{m}}
\newcommand{\bn}{\mathbf{n}}
\def\CA{{\mathcal A}}
\def\CB{{\mathcal B}}
\def\CC{{\mathcal C}}

\def\CL{{\mathcal L}}

\def\sC{{\mathsf C}}
\def\sM{{\mathsf M}}
\newcommand{\TF}{\texttt{F}}
\newcommand{\BA}{\mathbb{A}}
\newcommand{\BF}{\mathbb{F}}

\newcommand{\BZ}{{\mathbb Z}}
\newcommand{\fm}{{\mathfrak m}}

\newcommand{\kk}{\mathsf{k}}
\newcommand{\sff}{f}
\newcommand{\sg}{g}
\newcommand{\im}{\operatorname{Im}}
\newcommand{\id}[1]{\operatorname{id}^{#1}}

\newcommand{\Poi}[2]{\operatorname{P}^{#1}_{#2}}
\newcommand{\sS}{\scriptstyle}

\newcommand{\ges}{\operatorname{\sS\geqslant}}
\newcommand{\les}{\operatorname{\sS\leqslant}}

\newcommand{\rank}{\operatorname{rank}}%%%%%%%%%%%%%%%%%%%%%%%%%%%%%%%%%%%%%%%%%%%%%%%%%%%%%%%%%%%%%%%%%%%%%%%%%%%%%%%%
\newcommand{\HH}[2]{\operatorname{H}_{#1}(#2)}

\newcommand{\Ker}{\operatorname{Ker}}

\newcommand{\Span}{\operatorname{Span}}
\newcommand{\Imm}{\operatorname{Im}}

\newcommand{\xra}{\xrightarrow}

\newcommand{\Cone}{\operatorname{Cone}}

\theoremstyle{plain}
\newtheorem{theorem}{Theorem}[section]
\newtheorem{corollary}[theorem]{Corollary}
\newtheorem{lemma}[theorem]{Lemma}

\newtheorem{proposition}[theorem]{Proposition}

\theoremstyle{definition}

\newtheorem{remark}[theorem]{Remark}
\newtheorem{convention}[theorem]{Convention}
\newtheorem{definition}[theorem]{Definition}

\newtheorem*{goal}{Goal}
\numberwithin{equation}{theorem}
\theoremstyle{remark}

\begin{document}

\title[On a minimal free resolution over a local ring of codepth 3 of class $\clT$]{On a minimal free resolution of the residue field \\ over a local ring of codepth 3 of class $\clT$}

\author[V.~C.~Nguyen]{Van C.~Nguyen}
\address{Department of Mathematics, United States Naval Academy, Annapolis, MD 21402, U.S.A.}
\email{vnguyen@usna.edu}
\author[O.~Veliche]{Oana Veliche}
\address{Department of Mathematics, Northeastern University, Boston, MA 02115, U.S.A.}
\email{o.veliche@northeastern.edu}

\date{\today} 
\keywords{minimal free resolution, Koszul homology, class T, complete intersection, mapping cone}

\subjclass[2020]{13D02, 13D07, 13H10, 18G10, 18G35.}

%%%%%%%%%%%%%%%%%%%%%%%%%%%%%%%%%
%%%%%%%%%%%%%%%%%%%%%%%%%%%%%%%%%
\begin{abstract}
Let $R$ be any noetherian local ring with residue field $\kk$, and $A$ the homology of the Koszul complex on a minimal set of generators of the maximal ideal of $R$. In this paper, we show that a minimal free resolution of $\kk$ over $R$ can be obtained from a graded minimal free resolution of $\kk$ over $A$. More precisely, this is done by the iterated mapping cone construction, introduced by the authors in a previous work, using specific choices of ingredients. As applications, using this general perspective, we exhibit a minimal free resolution of $\kk$ over a complete intersection ring of any codepth, and explicitly construct a minimal free resolution of $\kk$ over a noetherian local ring of codepth 3 of class $\clT$ in terms of Koszul blocks.
\end{abstract}

\maketitle
%\thispagestyle{empt\title{}
%\tableofcontents

%%%%%%%%%%%%%%%%%%%%%%%%%%%%%%%
%%%%%%%%%%%%%%%%%%%%%%%%%%%%%%%
\section{Introduction}

Let $(R,\fm,\kk)$ be a noetherian local ring of codepth $c$ with maximal ideal $\fm$ and residue field $\kk$. Let $K$ be the Koszul complex on the minimal set of generators of $\fm$  and $A=\HH{}{K}$ its Koszul homology which has a graded-commutative $\kk$-algebra structure.

In \cite{G}, Golod proved that the Betti numbers of $\kk$ over $R$ have the fastest possible growth if and only if the Koszul homology algebra $A$ has trivial products and trivial Massey operations; such a ring is now called a Golod ring. Generalizing Golod's result, in \cite[Corollary 5.10]{Av74} Avramov proved that a minimal free resolution of $\kk$ over $R$ depends only on the ranks of the graded components $A_i$ and the multiplicative structure of $A$ with Massey products. In \cite{NV}, using the Koszul complex $K$ as building blocks, the authors explicitly constructed the minimal free resolution of $\kk$ over $R$ up to homological degree 5. As expected by the results in \cite{Av74}, the multiplicative structure of $A$ and the Massey products occurring in degree 4 play an important role in this construction. However, constructing the resolution beyond degree 5 requires a good understanding of higher-order Massey products, which remains elusive. On the other hand, using an iterated mapping cone construction in \cite{NV2}, we successfully described the entire free resolution of $\kk$ over a ring of class $\clC{c}$, i.e., complete intersection rings of codepth $c$, with $K$ as building blocks. This resolution is well-known, e.g. by Tate \cite[Theorem 4]{T} and studied by others. 

Since Massey products do not exist for rings with codepth $c \les 3$, this motivates the following. 

\begin{goal} 
For noetherian local rings $R$ of codepth at most 3, explicitly describe the minimal free resolution of $\kk$ as an $R$-module in terms of the Koszul complex $K$ and the multiplicative structure of Koszul homology $A$. 
\end{goal}

If $c=0$, then $R$ is regular and the resolution of $\kk$ is the Koszul complex $K$. If $c=1$, then $R$ is a hypersurface, that is of class $\clC{1}$, the resolution of $\kk$ is given in \cite{NV2}. If $c=2$, then $R$ is either complete intersection $\clC{2}$ or Golod, and the resolution for each case is  respectively described in \cite{NV2} and \cite{G}.

For $c=3$, local rings of codepth 3 are completely classified in terms of the multiplicative structure of $A$ as classes: $\clC{3}$, $\clT$, $\clB$, $\clG{r}$ with $r\ges 2$, and $\clH{p,q}$ with $p,q\ges 0$. This is due to the classification work of Weyman \cite{W} and Avramov-Kustin-Miller \cite{AKM}. In this case, the Poincar\'e series of $\kk$ over $R$ are explicitly given in terms of the multiplicative invariants of $A$; see \cite[Theorem 2.1]{Av2}, \cite[Corollary 4.6]{NV}. Recently in \cite{BCLP}, Briggs-Cameron-Letz-Pollitz used a general theory of twisted tensor products in the framework of universal resolutions (introduced by Priddy over Koszul algebras) to describe the minimal free resolution of $\kk$ over $R$. In their work, for a strictly Koszul homomorphism $\varphi: Q \to R$, one can obtain free resolutions over $R$ starting from free resolutions over $Q$ \cite[Theorem 6.5]{BCLP}. For local rings of codepth 3 or less, in terms of the classification types described in \cite{AKM}, every type except class $\clT$ is Cohen Koszul \cite[Theorem 3.10]{BCLP}. Consequently, using the twisted tensor products and $A_\infty$-structure, they gave a semifree resolution of $\kk$ over $R$ for these Cohen Koszul rings, except for class $\clT$ \cite[Theorem 7.8]{BCLP}. Observe that in \cite{BCLP}, the multiplicative structure of $A$ was not involved in their description of the resolution. 

Working toward the aforementioned goal, in this paper, we discuss how the mapping cone construction in general can yield a minimal free resolution of $\kk$ over any noetherian local ring $R$ of \emph{any codepth}. In particular, our main \Cref{graded res} provides a way to construct a minimal free resolution of $\kk$ as an $R$-module starting from a graded minimal free resolution of $\kk$ as a module over the graded-commutative ring $A$. In \Cref{ci}, we illustrate this result by reformulating the known resolution in \cite{NV2} for complete intersection rings $\clC{c}$ in this context. 

Among all the five classes of codepth 3 rings, the class $\clT$ has the closest multiplicative structure to a complete intersection ring. Therefore, we focus on class $\clT$ in \Cref{class T} as a new application of \Cref{graded res}. We explicitly construct the sequences of complexes and chain maps that satisfy the hypotheses of the main \Cref{graded res}, allowing us to form a graded resolution of $\kk$ over $A$; see \Cref{res k over A class T}. Through our approach, we are able to fill the missing case from \cite{BCLP} with a complete description of the resolution of $\kk$ over class $\clT$; see \Cref{res k over R class T}. To illustrate our results, in \Cref{sec:example} we discuss in detail an example of class $\clT$ that is a codepth 3 almost complete intersection ring and explain how to concretely describe the resolution and its differential maps using a finite set of data. 

We conclude the introduction with some details on how we construct a resolution of $\kk$ as a graded $A$-module for class $\clT$ in comparison with that for class $\clC{c}$. Indeed, our knowledge about complete intersection rings in \cite{NV2} proves to be essential in this work. First, from the classification results, the Koszul homology $A$ over class $\clT$ is a trivial extension of a $\kk$-subalgebra $B$ by a $B$-module $C$, where $B,C$ are described in \Cref{subsec:Koszul}. Here, $B$ looks like a truncated complete intersection. The essential ingredients to construct the differential maps of the graded resolution of $\kk$ over $A$ are the maps $\{\beta_k\}_{k \ges 0}$, $\{\beta_k'\}_{k \ges 0}$ and $\gamma_1, \gamma_2, \gamma_3$. The maps $\{\beta_k\}_{k\ges 0}$ are precisely the ones used in constructing the resolution in the case $\clC{3}$ (cf.~\Cref{def beta} vs. \Cref{def beta'}), whereas the maps $\{\beta'_{k}\}_{k\ges 0}$ behave like right inverse maps of $\{\beta_k\}_{k\ges 0}$, and the maps $\gamma_i$ relate to the component $C$. Second, our construction in \Cref{Candphi} is associated to a directed tree (cf.~\Cref{fig:tree}) which gives rise to a graded resolution of $\kk$ over $A$ for class $\clT$. In the $\clC{c}$ case, this tree is just a linear graph and it gives a linear graded minimal free resolution of $\kk$ over $A$, see \eqref{res k over A ci}. Finally, we observe that the graded resolution of $\kk$ over $A$ grows from the acyclic complex $\BA$ of $A$-modules in \eqref{BA}, which behaves like a \emph{seed} \eqref{seed} that generates the tree; see \Cref{fig:tree}.

%%%%%%%%%%%%%%%%%%%%%%%%%%%%%%%
%%%%%%%%%%%%%%%%%%%%%%%%%%%%%%%
\section{The minimal free resolution of the residue field}

In this paper, let $(R,\fm, \kk)$ be a noetherian local ring with maximal ideal $\fm$ and residue field $\kk=R/\fm$. For simplicity, we assume all local rings are noetherian. Let $n$ be the embedding dimension of $R$ and $c$ its codepth. Let $(K,\partial)$ be the Koszul complex on the minimal set of generators of $\fm$ and $A=A_0\oplus A_1\oplus \cdots \oplus A_c$ its Koszul homology. The differential graded structure of $K$ induces a graded-commutative $\kk$-algebra structure on $A$, where we view $A$ as a graded algebra with the grading given by the homological degrees. In general, $A$ is not a standard graded algebra. For example, in the codepth 3 case, $A$ is a standard graded algebra only when $R$ is a complete  intersection ring; see e.g. \cite{Av2}. As in \cite{NV, NV2}, we set 
\[ a_i\coloneqq \rank_{\kk} (A_i),\quad\text{for all } 0\les i\les c.\] 

The main result of this section is \Cref{graded res}. It provides a way to construct a minimal free resolution of the residue field $\kk$ over the local ring $R$ starting from a graded minimal free resolution of $\kk$ over the graded-commutative $\kk$-algebra $A$. We first set the following convention to be used throughout the paper. 

\begin{convention}
We denote the following maps by 
\[
\alpha \oplus \beta = \left(\begin{array}{c|c}\alpha&0\\ \hline 0&\beta \end{array}\right) \colon 
\xymatrixcolsep{9pt}
\xymatrixrowsep{6pt}
\xymatrix{
M \ar@{->}[r]^{\alpha} \ar@{}[d]|{\oplus}&N\ar@{}[d]|{\oplus}\\
M' \ar@{->}[r]^{\beta} &N'
}
\qquad \qquad
\begin{pmatrix}\alpha\\ \hline \beta\end{pmatrix} \colon 
\xymatrix{
M \ar@{->}[r]^{\alpha} \ar@{->}[rd]_{\beta} &N \ar@{}[d]|{\oplus} \\ &N'
} 
\qquad \qquad \text{ and } \qquad
(\alpha|\beta) \colon
\xymatrix{
M \ar@{->}[r]^{\alpha} \ar@{}[d]|{\oplus} & N. \\ M' \ar@{->}[ru]_{\beta} & 
}
\]
Likewise, for any integer $k\ges 1$ we denote by $\alpha^{k}\coloneqq \alpha^{k-1}\oplus \alpha: M^{\oplus k} \to N^{\oplus k}$ with $\alpha: M \to N$ appearing on the main diagonal $k$ times, and $0$ everywhere else. 

For any two complexes $(L,\partial^{L})$ and $(L',\partial^{L'})$ of modules over a ring, we denote by $(L\oplus L', \partial^{L}\oplus \partial^{L'})$ the direct sum of complexes. In particular, for any integer $k\ges 1$ we set $L^{k} \coloneqq L \oplus L \oplus \cdots \oplus L$ for a direct sum of $k$ copies of $L$. For any chain map $\lambda: L \to L'$, we denote $\lambda^k: L^k \to (L')^k$ for a direct sum of $k$ copies of $\lambda$ with respective components. 
We use the subscript to denote the homological component of $L$, and superscript with parentheses to label different complexes.

For a complex $L$ where $L_i=0$ for all $i <0$, we follow the convention in \cite{BH} and say $L$ is \emph{exact} if $\HH{i}{L}=0$ for all $i \ges 0$; and $L$ is \emph{acyclic} if $\HH{i}{L}=0$ for all $i > 0$.

For any complex $(L,\partial^{L})$ and any integer $s$, denote by $\Sigma^s L$ the complex with components $(\Sigma^s L)_i=L_{i-s}$ and differential maps $\partial^{\Sigma^s L}_i=(-1)^s \partial^L_{i-s}$ for all $i$. 

Let $L$ be a graded complex over a graded ring $S$ with $L_i = \bigoplus_{j \in \BZ} S^{b_{i,j}}(-j)$ for all $i$ and some non-negative integers $b_{i,j}$. Then, for $a\in \BZ$, $L(-a)$ denotes the graded complex with $L(-a)_i=\bigoplus_{s \in \BZ} S^{b_{i,j}}(-j-a)$.
\end{convention}

\medskip
The following is our main result, which shows that resolving $\kk$ over any local ring $R$ could amount to resolving $\kk$ over the Koszul homology $A$. 

\begin{theorem}
\label{graded res}
Let $(R,\fm,\kk)$ be a noetherian local ring of codepth $c \ges 1$, $(K, \partial)$ the Koszul complex of $R$ on a minimal set of generators of $\fm$, and $A= \bigoplus_{0 \les i \les c} A_i$ its Koszul homology which is a graded-commutative algebra with the grading given by the homological degrees.
Assume that there exists a graded minimal free resolution of $\kk$ as an $A$-module of the form 
\begin{equation*}
    \BF\colon \dots\to 
    \bigoplus_{s\ges k} A^{u_{k,s}}(-s)\xra{\partial_k^\BF} 
    \bigoplus_{s\ges k-1} A^{u_{k-1,s}}(-s)\xra{\partial_{k-1}^\BF}\dots \to \bigoplus_{s\ges 1} A^{u_{1,s}}(-s) \xra{\partial_{1}^\BF} A,
\end{equation*}
with differentials $\partial_k^\BF=[\varphi^{(k)}]$ for all integers $k \ges 1$, given by chain maps $\varphi^{(k)} \colon \sC^{(k)} \to \sC^{(k-1)}$, where $\sC^{(0)} \coloneqq K$ and $\sC^{(k)} \coloneqq \bigoplus_{s\ges k} \Sigma^s K^{u_{k,s}}$ with $\Imm \varphi^{(k)}\subseteq \fm\sC^{(k-1)}$. 

Then, the limit of the iterated mapping cone of the sequence of maps $\{\varphi^{(j)}\}_{j \ges 1}$ is a minimal free resolution of $\kk$ as an $R$-module.
\end{theorem}

\begin{proof} 
Consider the notation from \cite[Section 3]{NV2}. Let $\{\sM^{(j)},f^{(j)}\}_{j \ges 0}$ denote the iterated mapping cone sequence of $\{\varphi^{(j)}\}_{j \ges 1}$. Component wise, the $r$-th homological component of each complex $\sM^{(j)}$ is given by:
\[\sM^{(j)}_r=\sC^{(0)}_r\oplus\sC^{(1)}_{r-1}\oplus\dots\oplus \sC^{(j)}_{r-j},\quad\text{ for all } r\ges 0 .\]
Since $K_i=0$ for all $i<0$, we get by definition, $\sC^{(j)}_r=0$ for all $r<j$.

Let $F$ denote the limit mapping cone, as in \cite[Construction 3.1 and Definition 3.2]{NV2}.  
Component wise, the $r$-th homological component of $F$ is given by:
\[F_r=\sC^{(0)}_r\oplus\sC^{(1)}_{r-1}\oplus\dots\oplus \sC^{(j)}_{r-j}\oplus\dots,\quad\text{ for all } r\ges 0 .\]
By construction, $F$ is a complex, that is,  $\im(\partial_r^F) \subseteq \Ker(\partial_{r-1}^F)$ for all $r \ges 1$.

To show exactness of $F$ at each homological degree  $r \ges 1$, that is, the equality $\im(\partial_r^F)=\Ker(\partial_{r-1}^F)$ holds, we use induction on $r$ and an argument similar to the proof of \cite[Theorem 4.1]{NV2}. Remark that for each $r$, we have 
\[F_{r+1} = \sM^{(i)}_{r+1},\quad F_{r} = \sM^{(i)}_{r},\quad \text{and}\quad F_{r-1} = \sM^{(i)}_{r-1}, \text{ where } i=\left\lceil\frac{r}{2}\right\rceil+1. \] 
Therefore, the problem reduces to showing the exactness of the sequence 
\begin{equation}
\label{Mi}
\sM^{(i)}_{r+1}\to\sM^{(i)}_{r}\to\sM^{(i)}_{r-1}, \text{ for each } r \ges 1 \text{ and } i=\left\lceil\frac{r}{2}\right\rceil+1. 
\end{equation}

For $r=1$, we need to prove the exactness of the sequence $\sM^{(2)}_{2}\to\sM^{(2)}_{1}\to\sM^{(2)}_{0},$ that is,
\begin{equation}
\label{r1}
\xymatrix{
\xymatrixrowsep{0.01pc}
 \xymatrixcolsep{0.5pc}
 \sC^{(1)}_1\ar@{}[d]_{\bigoplus} \ar@{->}[dr]^{-\varphi^{(1)}_1}&&\\
 \sC^{(0)}_2\ar@{->}[r]^{\partial_2^{(0)}}&\sC^{(0)}_1\ar@{->}[r]^{\partial_1^{(0)}}&\sC^{(0)}_0
 } 
 \quad \text{ which is the same as } \quad 
 \xymatrix{
\xymatrixrowsep{0.01pc}
 \xymatrixcolsep{0.5pc}
 K_0^{u_{1,1}}\ar@{}[d]_{\bigoplus} \ar@{->}[dr]^{-\varphi^{(1)}_1}&&\\
 K_2\ar@{->}[r]^{\partial_2}&K_1\ar@{->}[r]^{\partial_1}&K_0.
 }
 \end{equation}
By hypothesis, the exactness of the augmented complex $\BF\to \kk \to 0$ at homological degree zero gives in particular the exactness of the sequence
\begin{equation}
\label{A1}
 0\to A_0^{u_{1,1}}\xra{[\varphi_1^{(1)}]} A_1\to 0.
\end{equation}
Let $z\in\Ker(\partial_1)$. Then $[z]\in A_1$.
By the exactness of \eqref{A1} there exists $z'\in K_0^{u_{1,1}}$ such that $[\varphi_1^{(1)}]([z'])=[z]$. That is, $z-\varphi_1^{(1)}(z') \in \Imm\partial_2$, so there exists $z''\in K_2$ such that \(z=\partial_2(z'') - \varphi_1^{(1)}(-z')\) which gives the exactness at $K_1$ in \eqref{r1}. Hence, $F$ is exact at degree 1. 

Next, we consider $r\ges 2$ and as above set \(i=\left\lceil\frac{r}{2}\right\rceil+1\). When $r$ is even we have $r=2i-2$ and when $r$ is odd we have $r=2i-3$. Therefore, the exactness of \eqref{Mi} is equivalent to the equalities:

\begin{equation}
\label{Mi2}
    \HH{2i-2}{\sM^{(i)}}=0=\HH{2i-3}{\sM^{(i)}},\quad\text{ for all $i\ges 2$}.
\end{equation}  
 %\cite[Lemma 3.6 and Theorem 3.7]{NV2}
Using \cite[Construction 3.1]{NV2} we set \[\sM^{(j)}\coloneqq \Cone(\psi^{(j)}),\quad\text{ for all } j\ges 1,\] where $\psi^{(1)}\coloneqq \varphi^{(1)}$ and the existence of $\psi^{(j)}$ is given by the following diagram with exact rows: 
\begin{equation*}
\label{cone diagram c}
\xymatrixrowsep{2pc}
 \xymatrixcolsep{3pc}
\xymatrix{
0\ar@{->}[r]&
\Sigma^{-1}\sM^{(j)}\ar@{->}[r]^{\Sigma^{-1}\sff^{(j)}}&
\Sigma^{-1}\sM^{(j+1)}\ar@{->}[r]^{\Sigma^{-1}\sg^{(j+1)}}&
\Sigma^{j} \sC^{(j+1)}\ar@{->}[d]^{\Sigma^{j}\varphi^{(j+1)}}\ar@{.>}[dl]_{\psi^{(j+1)}}\ar@{->}[r]&0\\
0\ar@{->}[r]&\sM^{(j-1)}\ar@{->}[r]^{\sff^{(j-1)}}&\sM^{(j)}\ar@{->}[r]^{\sg^{(j)}}&\Sigma^{j} \sC^{(j)}\ar@{->}[r]&0
}
\end{equation*}
which induce the rows of long exact sequences of homology that fit in \Cref{fig c}. 

From the definition of $\sC^{(j)}$, we have for each $v \in \BZ$
\begin{equation*}
    \HH{v}{\sC^{(j)}}=\HH{v}{\bigoplus_{s\ges j} \Sigma^s K^{u_{j,s}}}\cong \bigoplus_{s\ges j}\HH{v}{\Sigma^s K^{u_{j,s}}}=
    \bigoplus_{s\ges j}\HH{v-s}{K^{u_{j,s}}}=\bigoplus_{s\ges j} A^{u_{j,s}}_{v-s},
\end{equation*}
hence, by the hypothesis $[\varphi^{(j)}]=\partial_j^\BF$ for all $j \ges 1$, each $j$-th homological component of $\BF$ is a graded $A$-module of the form
\begin{equation*}
\BF_j=\bigoplus_{v\in\BZ}\HH{v}{\sC^{(j)}}.
\end{equation*}
By the acyclicity of $\BF$, we obtain that for each $v \in \BZ$, there exists an exact sequence of homologies:
\begin{equation*}
    \dots\to \HH{v}{\sC^{(j)}}\xra{[\varphi^{(j)}]} \HH{v}{\sC^{(j-1)}}\xra{[\varphi^{(j-1)}]} \dots \to \HH{v}{\sC^{(1)}}\xra{[\varphi^{(1)]}} \HH{v}{\sC^{(0)}}. 
\end{equation*}
Therefore, the vertical sequences in \Cref{fig c} are exact. This is all that is needed in the proof by induction of \cite[Lemma 3.6]{NV2} to conclude that $[f^{(j)}]=0$ for all $j \ges 0$. In particular, all rows of the long exact sequences in \Cref{fig c} split. 
Using this result, as in the proof of \cite[Theorem 3.7]{NV2}, we conclude that for all $v \ges 0$ and $j \ges 1$,   
\begin{equation*}
\label{eq H}
\HH{v}{\sM^{(j)}} \cong \begin{cases}
\HH{v}{\sM^{(j-1)}}, & \quad 0\les v \les 2j-2 \text{ or } v \ges 2j+c+1 \\
0, & \quad v=2j-1 \text{ or } v=2j.
\end{cases}
\end{equation*}
This gives us the desired expression
\[\HH{2i-2}{\sM^{(i)}}\cong \HH{2i-2}{\sM^{(i-1)}}=0,\] where the first isomorphism follows from the first case with $j=i$ and $v=2i-2$, and the second equality follows from the second case with $j=i-1$ and $v=2(i-1)$. Similarly, we obtain
\[\HH{2i-3}{\sM^{(i)}}\cong \HH{2i-3}{\sM^{(i-1)}}=0,\] where the first isomorphism follows from the first case with $j=i$ and $v=2i-3$, and the second equality follows from the second case with $j=i-1$ and $v=2(i-1)-1$. Hence \eqref{Mi2} holds, so $F$ is acyclic.

The minimality of $F$ follows from the condition  $\Imm \varphi^{(k)}\subseteq \fm\sC^{(k-1)}$ for all $k\ges 1$ and from the fact that the differential of $K$ satisfies by definition $\partial_i\subseteq \fm\partial_{i-1}$ for all $i\ges 1.$
\end{proof}

\begin{remark}
All superscripts in \Cref{fig c} should have parentheses around them to label different complexes or maps, we omit writing the parentheses there for simplicity of notation. 
\end{remark}

\begin{landscape}
\begin{figure}[h]
 \caption{Long exact sequences of homology via the iterated mapping cone construction}
 \label{fig c}
 \centering
\xymatrixrowsep{3.3pc}
 \xymatrixcolsep{2.8pc}
\small{ \xymatrix
 {
 \vdots
 \ar@{->}[d]_{[\varphi^{j+2}]}
 &&& 
 \vdots\ar@{->}[d]
 &&&
  \vdots\ar@{->}[d]
 \\
 \HH{u-j+1}{\sC^{j+1}} 
 \ar@{->}[d]_{[\varphi^{j+1}]} 
 \ar@{->}[r]^{[\psi^{j+1}]}  
 &
 \HH{u+1}{\sM^j} 
 \ar@{->}[r]^{[\sff^j]}
 \ar@{.>}[ld]_{[\sg^{j}]}   
 & 
 \HH{u+1}{\sM^{j+1}}
 \ar@{->}[r]^{[\sg^{j+1}]} 
 &
 \HH{u-j}{\sC^{j+1}}
 \ar@{->}[d]_{[\varphi^{j+1}]} 
 \ar@{->}[r]^{[\psi^{j+1}]} \ar@{.>}[ld]_{[\psi^{j+1}]} 
 &
 \HH{u}{\sM^j}
 \ar@{->}[r] ^{[\sff^j]}
 \ar@{.>}[ld]_{[\sg^{j}]} 
 &
 \HH{u}{\sM^{j+1}} \ar@{->}[r]^{[\sg^{j+1}]} &
 \HH{u-j-1}{\sC^{j}} \ar@{->}[d]_{[\varphi^{j}]} 
 \\
 \HH{u-j+1}{\sC^j} 
 \ar@{->}[d] _{[\varphi^{j}]} 
 \ar@{->}[r]^{[\psi^{j}]} 
 &
 \HH{u}{\sM^{j-1}} 
 \ar@{->}[r]^{[\sff^{j-1}]} 
 \ar@{.>}[ld]_{[\sg^{j-1}]} 
 & 
 \HH{u}{\sM^{j}} 
 \ar@{->}[r]^{[\sg^{j}]}  
 &
\HH{u-j}{\sC^j} 
 \ar@{->}[d]_{[\varphi^{j}]} 
 \ar@{->}[r]^{[\psi^{j}]}
 \ar@{.>}[ld]_{[\psi^{j}]}
 &
 \HH{u-1}{\sM^{j-1}} 
 \ar@{->}[r]^{[\sff^{j-1}]}
 \ar@{.>}[ld]_{[\sg^{j-1}]} 
 &
 \HH{u-1}{\sM^{j}} 
 \ar@{->}[r]^{[\sg^{j}]}
 &
 \HH{u-j-1}{\sC^{j-1}}
 \ar@{->}[d]_{[\varphi^{j-1}]} 
 \ar@{.>}[ld]_{[\psi^{j}]}  
 \\
 \HH{u-j+1}{\sC^{j-1}}  
 \ar@{->}[d] 
 \ar@{->}[r]^{[\psi^{j-1}]} 
 &
 \HH{u-1}{\sM^{j-2}} 
 \ar@{->}[r]^{[\sff^{j-2}]}
 \ar@{.>}[ld]_{[\sg^{j-2}]} 
 & 
 \HH{u-1}{\sM^{j-1}} 
 \ar@{->}[r]^{[\sg^{j-1}]}  
 &
 \HH{u-j}{\sC^{j-1}} 
 \ar@{->}[d]_{[\varphi^{j-1}]} 
 \ar@{->}[r]^{[\psi^{j-1}]}
 \ar@{.>}[ld]_{[\psi^{j-1}]}  
 &
 \HH{u-2}{\sM^{j-2}} 
 \ar@{->}[r]^{[\sff^{j-2}]}
 \ar@{.>}[ld]_{[\sg^{j-2}]} 
 &
 \HH{u-2}{\sM^{j-1}} 
 \ar@{->}[r] ^{[\sg^{j-1}]}  
 &
 \HH{u-j-1}{\sC^{j-2}}
 \ar@{->}[d]_{[\varphi^{j-2}]}
 \ar@{.>}[ld]_{[\psi^{j-1}]} 
 \\
 \vdots \ar@{->}[r] &
 \HH{u-2}{\sM^{j-3}} 
 \ar@{->}[r]^{[\sff^{j-3}]} 
 & 
 \HH{u-2}{\sM^{j-2}} 
 \ar@{->}[r]^{[\sg^{j-2}]}  
 &
 \HH{u-j}{\sC^{j-2}} 
 \ar@{->}[d] 
 \ar@{->}[r]^{[\psi^{j-2}]}
 &
 \HH{u-3}{\sM^{j-3}} 
 \ar@{->}[r]^{[\sff^{j-3}]} 
 \ar@{.>}[ld]_{[\sg^{j-3}]} 
 &
 \HH{u-3}{\sM^{j-2}} 
 \ar@{->}[r]^{[\sg^{j-2}]}   
 &
 \HH{u-j-1}{\sC^{j-3}} 
 \ar@{->}[d]_{[\varphi^{j-3}]} 
 \ar@{.>}[ld]_{[\psi^{j-2}]} 
 \\
 &&&\vdots \ar@{->}[r] 
 &
 \HH{u-4}{\sM^{j-4}} 
 \ar@{->}[r] ^{[\sff^{j-4}]}
 &
 \HH{u-4}{\sM^{j-3}} 
 \ar@{->}[r]^{[\sg^{j-3}]}   &
  \HH{u-j-1}{\sC^{j-4}} 
 \ar@{->}[d] \\
 &&&&&&\vdots
 }
 }
\end{figure}
\end{landscape}

\begin{corollary}
\label{cor: res over R}
    The minimal free resolution $F$ of $\kk$ over $R$ from \Cref{graded res} has the following general description. For any $j \ges 0$, 
\[
        F_{2j} = \sC^{(0)}_{2j}\oplus \sC^{(1)}_{2j-1}\oplus\dots\oplus \sC^{(j)}_{j} \qquad \text{ and } \qquad 
        F_{2j+1}= \sC^{(0)}_{2j+1}\oplus \sC^{(1)}_{2j}\oplus\dots\oplus \sC^{(j)}_{j+1}. 
 \]
The differential maps $\partial^F_{2j}: F_{2j} \to F_{2j-1}$ are of size $j \times (j+1)$ as blocks of maps, 
    \begin{align*}
     \partial^F_{2j}&=
     \begin{pmatrix}
     \partial^{(0)}_{2j}&\varphi^{(1)}_{2j-1} &0&0&\dots&0&0\\[0.15in]
                   0&\partial^{(1)}_{2j-1}&\varphi^{(2)}_{2j-2}&0&\dots&0&0\\[0.15in]   
                   \vdots&\vdots&\vdots&\vdots&\vdots&\vdots&\vdots\\[0.15in]
                   0&0&0&0&\dots&\varphi^{(j-1)}_{j+1}&0\\[0.15in]
                   0&0&0&0&\dots&\partial^{(j-1)}_{j+1}&\varphi^{(j)}_{j}\\[0.15in]                   
     \end{pmatrix},
     \end{align*}
     and $\partial^F_{2j+1}: F_{2j+1} \to F_{2j}$ are of size $(j+1) \times (j+1)$ as blocks of maps,
     \begin{align*}
      \partial^F_{2j+1}&=
     \begin{pmatrix}
     \partial^{(0)}_{2j+1}& \varphi^{(1)}_{2j} &0&0&\dots&0&0\\[0.15in]
                   0&\partial^{(1)}_{2j}&\varphi^{(2)}_{2j-1}&0&\dots&0&0\\[0.15in]   
                   \vdots&\vdots&\vdots&\vdots&\vdots&\vdots&\vdots\\[0.15in]
                   0&0&0&0&\dots&\varphi^{(j-1)}_{j+2}&0\\[0.15in]
                   0&0&0&0&\dots&\partial^{(j-1)}_{j+2}&\varphi^{(j)}_{j+1}\\[0.15in]
                   0&0&0&0&\dots&0&\partial^{j}_{j+1}\\[0.15in]                   
     \end{pmatrix},
    \end{align*}
where the differential maps $\partial^{(k)}$ in the $\sC^{(k)}$ complex appear on the main diagonal, and the chain maps $\varphi^{(k)}$ appear on the superdiagonal (directly above the main diagonal).   
\end{corollary}

\begin{remark} 
If the hypotheses of \Cref{graded res} are satisfied for a local ring $R$ of codepth $c \ges 1$, then the description of $F$ holds as in \Cref{cor: res over R}. For a particular ring $R$, it is possible to find a minimal free resolution of $\kk$ over the Koszul homology $A$, however it is not trivial to describe specific $u_{k,s}$ and $\varphi^{(k)}_i: \sC_i^{(k)} \to \sC_i^{(k-1)}$ such that the hypotheses of \Cref{graded res} are satisfied. We believe this can be done for local rings of codepth 3, since the algebra structure of $A$ is well understood by the classification from \cite{AKM,W}. In the next sections, we provide explicit $u_{k,s}$, $\varphi^{(k)}_i$, $\sC_i^{(k)}$, and $\BF$ satisfying the hypotheses of \Cref{graded res} for the complete intersection class of codepth $c \ges 1$ and for the class $\clT$ of codepth $3$.
\end{remark}

\begin{definition} 
\label{gr Poi}
If a graded minimal free resolution of $\kk$ as an $A$-module has the form
\begin{equation*}
    \BF\colon \dots\to \bigoplus_{s\ges 0} A^{u_{k,s}}(-s)\xra{\partial_k^\BF} \bigoplus_{s\ges 0} A^{u_{k-1,s}}(-s)\xra{\partial_{k-1}^\BF}\dots \to \bigoplus_{s\ges 0} A^{u_{1,s}}(-s) \xra{\partial_{1}^\BF} A
\end{equation*}
then the \emph{graded Poincar\'e series of $A$} is defined as:
    \begin{equation*}
        \Poi{A}{\kk}(t,z)=\sum_{k,s\ges 0}u_{k,s}t^kz^s.
    \end{equation*}
\end{definition}

\begin{remark}   
If for each $k\ges 0$ we have $u_{k,s}=0$ for all $s \gg 0$,  then $u_k\coloneqq \sum_{s\ges 0}u_{k,s}$ is called  the {\it $k$-th Betti number} of $\kk$ as an $A$-module.
In this case, the \emph{Poincar\'e series of $A$} is:
\begin{equation*}
    %\label{Poi}
        \Poi{A}{\kk}(t) \coloneqq \sum_{k \ges 0}u_{k}t^k=\Poi{A}{\kk}(t,1).
    \end{equation*}
Note that in the literature, see e.g. \cite[1.7]{Av2}, $\Poi{A}{\kk}(t)$ is defined as $\Poi{A}{\kk}(t)\coloneqq \Poi{K}{\kk}(t)$.    
\end{remark}

\begin{corollary} 
\label{Poi cor}
Under the hypotheses of \Cref{graded res} and let $n$ be
the embedding dimension of $R$, we have $\Poi{R}{\kk}(t)=(1+t)^n\Poi{A}{\kk}(t,t)$.
\end{corollary}

\begin{proof}
Note that by definition, $\sC^{(j)}_r=0$ for all $r<j.$  The rank of each free $R$-module $F_i$ is given by:
\begin{align*}
\rank_R F_i&= \sum_{k \ges 0}\rank_R \sC^{(k)}_{i-k} =  \sum_{k,s\ges 0} \rank_RK^{u_{k,s}}_{i-k-s}=\sum_{k,s\ges 0} \binom{n}{i-k-s}u_{k,s}=  \sum_{j\ges 0} \binom{n}{i-j}\sum_{k+s=j}u_{k,s}.
                \end{align*}
Thus, we get
    \begin{align*}
    \Poi{R}{\kk}(t)&=\sum_{i\ges 0} (\rank_{R}F_i)t^i =\sum_{i\ges 0}\sum_{j\ges 0} \binom{n}{i-j}t^{i-j}\sum_{k+s=j}u_{k,s}t^kt^s=(1+t)^n\Poi{A}{\kk}(t,t).
    \end{align*}
\end{proof}

%%%%%%%%%%%%%%%%%%%%%%%%%%%%%%%
%%%%%%%%%%%%%%%%%%%%%%%%%%%%%%%
\section{Application on class $\clC c$}
\label{ci}

The free resolution of the residue field $\kk$ over a complete intersection ring is well-known, see e.g. by Tate \cite[Theorem 4]{T}. In \cite{NV2}, the authors also obtained this resolution using the iterated mapping cone construction and the short exact sequences of Koszul homologies \cite[Theorem 4.1]{NV2}. One advantage of this approach is the description of the differential maps of the resolution is given explicitly as blocks of maps that are easily obtained from a finite set of data. The second advantage is that this method extends to other classes of rings as we will see in the subsequent sections.
We reformulate here some results from \cite{NV2} to better fit the scope of the present paper and to illustrate \Cref{graded res}.

For the rest of this section, let $(R,\fm,\kk)$ be a complete intersection local ring of codepth $c \ges 1$. Let $(K,\partial)$ be the Koszul complex of $R$ on a minimal set of generators of $\fm$ and $A=A_0\oplus A_1\oplus \cdots \oplus A_c$ its homology. We denote by $\{[z^1_1],\dots,[z^1_c]\}$ a $\kk$-basis for $A_1$, where each $z^1_u$ is a cycle representative in $K_1$. It is known that $A=\bigwedge A_1$, the exterior algebra on degree 1.

\begin{definition}
\label{def beta} 
For any integer $k\ges 0$, set $b_k=\textstyle\binom{k+c-1}{c-1}$. A set with $b_k$ elements is indexed by words $u_1\dots u_k$ of length $k$ with 
$1\les u_1\les \dots \les u_k\les c$ that are ordered lexicographically.

Let $\beta_k$ be the $b_{k-1}\times b_{k}$ matrix with entries in $K_1$, given as follows. For $1\les v_1\les\dots\les v_{k-1}\les c$ and $1\les u_1\les\dots\les u_{k}\les c$, its entry in the position $(v_1\dots v_{k-1},u_1\dots u_{k})$ is given by  
\begin{align*}
 \beta_k(v_1\dots v_{k-1},u_1\dots u_{k}) \coloneqq \begin{cases}
 z^1_{u},
 &\text{ if } u_1\dots u_{k}=[v_1\dots v_{k-1}u], \text{ for } 1 \les u \les c\\
 0,
 &\text{ otherwise.} 
 \end{cases}
\end{align*}
Here, we adopt the notation from \cite[Section 2]{NV2}. When the index is not in non-decreasing order, we use the square bracket notation $[u_1\dots u_k]$ to reorder it non-decreasingly; for example $[1312]=1123$. 

For $k\ges 0$, let $[\beta_k]$ denote the matrix with entries in $A_1$ where we replace entries $z^1_u$ in $\beta_k$ by $[z^1_u]$. 
\end{definition}

\begin{remark}
It is straightforward that $\beta_k$ is the matrix representation of the map $\zeta^{k-1}$ given in \cite[(2.1.1)]{NV2}. 
\end{remark}

\begin{proposition} 
\label{graded res CI}
Let $(R,\fm,\kk)$ be a complete intersection local ring of codepth $c \ges 1$, $K$ the Koszul complex of $R$ on a minimal set of generators of $\fm$, and $A$ its Koszul homology which we consider as a graded-commutative ring. Then, a graded minimal resolution of $\kk$ over $A$ is given by:
\begin{equation}
\label{res k over A ci}
   \BF\colon     \cdots \to A^{b_{k}}(-k)\xra{[\beta_k]}A^{b_{k-1}}(-k+1)\to \cdots \to A^{b_2}(-2)\xra{[\beta_2]}A^{b_1}(-1)\xra{[\beta_1]} A^{b_0}\to 0,
\end{equation}
where $b_k=\textstyle\binom{k+c-1}{c-1}$ and matrices $\beta_k$ are defined in \Cref{def beta}, for each integer $k\ges 0$.
\end{proposition}

\begin{proof} 
Let $\CA_0$ be the complex $0 \to A_0^{b_0} = \kk \to 0$. For each $k \ges 1$, by \cite[Proposition 2.2]{NV2}, we consider the exact complex of $\kk$-modules of the form
\begin{equation}
\label{ses codepth c}
\CA_k \colon\qquad  0\to A_0^{b_{k}}\xrightarrow{[\beta_{k}]}A_1^{b_{k-1}} \xrightarrow{[\beta_{k-1}]}A_2^{b_{k-2}} \xrightarrow{[\beta_{k-2}]} \cdots A_{c-1}^{b_{k-c+1}} \xrightarrow{[\beta_{k-c+1}]}A_c^{b_{k-c}} \to 0,
\end{equation}
where the component $A_0^{b_k}$ is in the homological degree $k\ges 0$.
It follows that $\BF=\bigoplus_{k \ges 0} \CA_k$. 
Indeed, for each $i\ges 0$, we have
\[ \BF_i=\Big(\bigoplus_{k\ges 0} \CA_k \Big)_i= \bigoplus_{k\ges 0} A_{k-i}^{b_i}=A^{b_i}(-i).\]

The exactness of $\CA_k$ for each $k\ges 1$ implies the acyclicity of $\BF$. 
Moreover, from the above description, it is clear that each $\BF_i$ is free as an $A$-module. The minimality of $\BF$ follows since the differential maps all have entries in $A_1$. 
\end{proof}

We recover the well-known Poincar\'e series for complete intersection rings. 
\begin{corollary}
Let $(R,\fm,\kk)$ be a complete intersection local ring of embedding dimension $n$, codepth $c \ges 1$, and $A$ its Koszul homology. Then, we have
\begin{equation*}
    \Poi{A}{\kk}(t,z)=\frac{1}{(1-tz)^c} \qquad\text{ and } \qquad \Poi{R}{\kk}(t)=\frac{(1+t)^n}{(1-t^2)^c}.
\end{equation*}
\end{corollary}

\begin{proof}
By \Cref{gr Poi}, \Cref{graded res CI}, and definition of $b_k$, we have
    \begin{align*}
    \Poi{A}{\kk}(t,z)=\sum_{k\ges 0} b_kt^kz^k=\sum_{k\ges 0} \binom{k+c-1}{c-1}(tz)^k=\frac{1}{(1-tz)^c}.
\end{align*}
The expression for $\Poi{R}{\kk}(t)$ follows from \Cref{Poi cor}.
\end{proof}

Using this reformulation, we recover \cite[Theorem 4.1]{NV2} as a consequence of \Cref{graded res} and \Cref{graded res CI}.

\begin{corollary}
    Let $(R,\fm,\kk)$ be a complete intersection local ring of codepth $c \ges 1$ and $(K,\partial)$ the Koszul complex of $R$ on a minimal set of generators of $\fm$. Then, the minimal free resolution $(F,\partial^F)$ of the residue field $\kk$ over $R$ have the components and differential maps given by:
    \begin{equation*}
     F_i=K_i^{b_0} \oplus K^{b_1}_{i-2}\oplus K^{b_2}_{i-4}\oplus\cdots \oplus K^{b_j}_{i-2j}\oplus\cdots \qquad \text{ and } \qquad
 \partial_i^F=
 \begin{pmatrix}
 \partial^{b_0}_i & \beta_1& 0                &           0       &0&\cdots\\[0.2cm]
           0  & \partial_{i-2}^{b_1} &\beta_2&           0       &0&\cdots\\[0.2cm]
           0  &                 0 & \partial_{i-4}^{b_2} &\beta_3&0&\cdots\\[0.2cm]
      \vdots  &           \vdots  &          \vdots  &           \ddots  &\ddots  &
 \end{pmatrix},
 \end{equation*}
for all integers $i\ges 0$, where $b_k=\textstyle\binom{k+c-1}{c-1}$ and the matrices $\beta_k$ are defined in \Cref{def beta}, for all integers $k\ges 0$.
\end{corollary}

%%%%%%%%%%%%%%%%%%%%%%%%%%%%%%%
%%%%%%%%%%%%%%%%%%%%%%%%%%%%%%%
\section{Application on class $\clT$}
\label{class T}

In this section, let $(R,\fm,\kk)$ be a local ring of embedding dimension $n$ and codepth $c=3$ of class $\clT$. This is one of the five classes in the classification work of \cite{Av2, AKM, W}. We construct a sequence of complexes $\{\sC^{(k)}\}_{k \ges 0}$ (see \eqref{def Ck}) and a sequence of chain maps $\{\varphi^{(k)} \colon \sC^{(k)} \to \sC^{(k-1)}\}_{k \ges 1}$ (see \eqref{matrix phi}) that satisfy the hypotheses of \Cref{graded res}.  Consequently, they give rise to a graded minimal resolution of $\kk$ over $A$ (see \Cref{res k over A class T}), and a minimal free resolution of $\kk$ over $R$ (see \Cref{res k over R class T}). We explicitly describe the components and differential maps of these resolutions.

%%%%%%%%%%%%%%%%%%%%%%%%%%%%%%%
\subsection{Exactness of sequences on Koszul homologies} 
\label{subsec:Koszul}

For class $\clT$, the Koszul homology components $A_i$, for $0\les i\les 3$, have the following $\kk$-bases:
\begin{align*}
A_0 \colon&\ [1], \\
A_1 \colon&\ [z^1_1],\, [z^1_2],\, [z^1_3], \dots, [z^1_{a_1}], \\
A_2 \colon&\ [z^1_1]\cdot[z^1_2],\, [z^1_2]\cdot[z^1_3],\, [z^1_1]\cdot[z^1_3],\quad\text{and}\quad [z^2_{1}],\dots,[z^2_{a_2-3}], \\
A_3 \colon&\ [z^3_1],\dots, [z^3_{a_3}],
\end{align*}
where each $z^j_u$ is a cycle representative in $K_j$. Here, $A_1\cdot A_1=\Span_\kk\{[z^1_1]\cdot[z^1_2],\, [z^1_2]\cdot[z^1_3],\, [z^1_1]\cdot[z^1_3]\}$ and all other products in $A$ are zero.

Let $B_i,C_j$ be $\kk$-vector spaces with the following bases
\begin{align*}
B_0\colon&\ [1],&& \\
B_1\colon&\  [z^1_1],\, [z^1_2],\, [z^1_3], & C_1\colon&\ [z^1_4],\dots, [z^1_{a_1}], \\
B_2\colon&\ [z^1_1]\cdot[z^1_2],\, [z^1_2]\cdot[z^1_3],\, [z^1_1]\cdot[z^1_3], & C_2\colon&\ [z^2_1],\dots, [z^2_{a_2-3}],\\
&&C_3\colon&\ [z^3_1],\dots, [z^3_{a_3}].
\end{align*}
Let $B \coloneqq B_0\oplus B_1\oplus B_2$ and $C \coloneqq C_1\oplus C_2\oplus C_3$. It is straightforward that $B$ is a $\kk$-subalgebra of $A$ and $A$ is a trivial extension of $B$ with $C$, that is, $A\cong B\ltimes C$.

\begin{definition}
\label{def beta'}
Similar to \Cref{def beta}, for any integer $k\ges 0$, set $b_k=\textstyle\binom{k+2}{2}$ and let $\beta_k$ be the $b_{k-1}\times b_{k}$ matrix with entries in the subset $\{0, z^1_1,z^1_2,z^1_3\}$ of $K_1$ given as follows. For $1\les v_1\les\dots\les v_{k-1}\les 3$ and $1\les u_1\les\dots\les u_{k}\les 3$, its entry in the position $(v_1\dots v_{k-1},u_1\dots u_{k})$ is given by 
\begin{align*}
\beta_k(v_1\dots v_{k-1},u_1\dots u_{k}) \coloneqq \begin{cases}
 z^1_{u},
 &\text{ if } u_1\dots u_{k}=[v_1\dots v_{k-1}u], \text{ for } 1 \les u \les 3\\
 0,
 &\text{ otherwise.} 
 \end{cases}
      \end{align*}

For any integer $k\ges 0$, let $\beta_k'$ be the $b_{k-1}\times b_{k-2}$ matrix with entries in the subset $\{0, z^1_1\wedge z^1_2,z^1_2\wedge z^1_3,z^1_1\wedge z^1_3\}$ of $K_2$ given as follows. For $1\les u_1\les\dots\les u_{k-1}\les 3$ and $1\les v_1\les\dots\les v_{k-2}\les 3$, its entry in the position 
$(u_1\dots u_{k-1},v_1\dots v_{k-2})$ is given by
\begin{align*}
\beta_k'(u_1\dots u_{k-1},v_1\dots v_{k-2}) \coloneqq \begin{cases}
 z^1_{u'}\wedge z^1_{u''},
 &\text{ if } u_1\dots u_{k-1}=[v_1\dots v_{k-2}u]\\
 0,
 &\text{ otherwise.} 
 \end{cases}
      \end{align*}
Here, for $u\in\{1,2,3\}$ we denote $\{u',u''\}=\{1,2,3\}\setminus\{u\}$ with $u'< u''$. As before, replacing the entries by their corresponding homology, we denote $[\beta_k]$ and $[\beta_k']$ for matrices with entries in $A_1$ and $A_2$ respectively, and  $[\beta_k]$ and $[\beta_k']$ can be considered as maps of degree 1 and 2 respectively. 
\end{definition}

\begin{remark}
\label{rmk: ex beta}
We illustrate the above definitions for the first few values of $k$:
\begin{align*}
\beta_0&=0, & \beta_0'&=0, \\
\beta_1&=\begin{pmatrix}z^1_1 & z^1_2 & z^1_3\end{pmatrix},
& \beta_1'&=0, \\
\beta_2&=\begin{pmatrix} z^1_1 & z^1_2 & z^1_3 & 0 & 0 & 0 \\
    0 & z^1_1 & 0 & z^1_2 & z^1_3 & 0 \\
    0 & 0 & z^1_1 & 0 & z^1_2 & z^1_3 
\end{pmatrix}, 
& \beta_2'&= \begin{pmatrix}
    z^1_2 \wedge z^1_3 \\ z^1_1 \wedge z^1_3 \\ z^1_1 \wedge z^1_2
\end{pmatrix}, \\
\beta_3&=\begin{pmatrix} 
    z^1_1 & z^1_2 & z^1_3 & 0 & 0 & 0 & 0 & 0 & 0 & 0 \\
    0 & z^1_1 & 0 & z^1_2 & z^1_3 & 0 & 0 & 0 & 0 & 0 \\
    0 & 0 & z^1_1 & 0 & z^1_2 & z^1_3 & 0 & 0 & 0 & 0 \\
    0 & 0 & 0 & z^1_1 & 0 & 0 & z^1_2 & z^1_3 & 0 & 0 \\
    0 & 0 & 0 & 0 & z^1_1 & 0 & 0 & z^1_2 & z^1_3 & 0 \\
    0 & 0 & 0 & 0 & 0 & z^1_1 & 0 & 0 & z^1_2 & z^1_3 \\
\end{pmatrix},
& \beta_3'& = \begin{pmatrix}
    z^1_2 \wedge z^1_3 & 0 & 0 \\
    z^1_1 \wedge z^1_3 & z^1_2 \wedge z^1_3 & 0 \\
    z^1_1 \wedge z^1_2 & 0 & z^1_2 \wedge z^1_3 \\
    0 & z^1_1 \wedge z^1_3 & 0 \\
    0 & z^1_1 \wedge z^1_2 & z^1_1 \wedge z^1_3 \\
    0 & 0 & z^1_1 \wedge z^1_2
\end{pmatrix}.
\end{align*}
\end{remark}

\begin{proposition}
\label{exact B}
Retain the above notation. For any integer $k\ges 1$, there exists an exact complex of $\kk$-modules of the form:
\begin{equation*}
 \CB_k\colon \xymatrixrowsep{1.5pc}
 \xymatrixcolsep{2pc}
 \xymatrix{
 0\ar@{->}[r]^{} 
 &B_0^{b_{k}}\ar@{->}[r]^{[\beta_k]}
 &B_1^{b_{k-1}}\ar@{->}[r]^{[\beta_{k-1}]}\ar@{}[d]_{\bigoplus} 
 &B_2^{b_{k-2}}\ar@{->}[r]^{}
 &0,
 \\
 &           
 &B_0^{b_{k-3}}\ar@{^{(}->}[ru]_{[\beta_{k-1}']}           
 &
 &
 }
 \end{equation*}
with $B_0^{b_k}$ is placed in the homological position $k$ in the complex $\CB_k$.
\end{proposition}

\begin{proof}
By \cite[Proposition 2.2]{NV2} with $c=3$, for each integer $k\ges 1$ there exists a split exact sequence of $\kk$-modules:
\begin{equation*}
    \xymatrixrowsep{1.5pc}
 \xymatrixcolsep{2pc}
 \xymatrix{
 0\ar@{->}[r]^{} 
 &(\bigwedge^0 B_1)^{b_k}\ar@{->}[r]^{[\beta_k]}
 &(\bigwedge^1 B_1)^{b_{k-1}}\ar@{->}[r]^{[\beta_{k-1}]}
 &(\bigwedge^2 B_1)^{b_{k-2}}\ar@{->}[r]^{[\beta_{k-2}]}
 &(\bigwedge^3 B_1)^{b_{k-3}}\ar@{->}[r]^{}
 &0.
 }
\end{equation*}
To get the desired short exact sequence, using $B_0\cong\bigwedge^0 B_1\cong \bigwedge^3 B_1$ and  $B_2\cong\bigwedge^2 B_1$, it is enough to check that for $k\ges 3$, the map $[\beta_{k-1}']$ is a right inverse of the map $[\beta_{k-2}]$. That is, 
for all $k\ges 1$, we check  $\beta_{k}\beta_{k+1}'=z^1_1\wedge z^1_2\wedge z^1_3\, I_{b_{k-1}}$ in $K_3$, where $I$ denotes the identity matrix.

Indeed, for $1\les v_1\les\dots \les v_{k-1}\les 3$ and $1\les v_1'\les\dots\les v_{k-1}'\les 3$, we calculate the entry $(v_1\dots v_{k-1},v_1'\dots v_{k-1}')$ of the product of matrices $\beta_{k}\beta_{k+1}'$. By definitions, the entries of $\beta_{k}$ and $\beta_{k+1}'$ are nonzero when $u_1\dots u_k = [v_1 \dots v_{k-1}u] = [v_1' \dots v_{k-1}'v]$, for some $u,v \in \{1,2,3\}$. 

If $v_1\dots v_{k-1} = v_1'\dots v_{k-1}'$, then the entry in $\beta_{k}\beta_{k+1}'$ is given by 
\[
    \sum_{u=1}^3 z^1_u\wedge z^1_{u'}\wedge z^1_{u''} =z^1_1\wedge z^1_2\wedge z^1_3 + z^1_2\wedge z^1_1\wedge z^1_3 + z^1_3\wedge z^1_1\wedge z^1_2 = z^1_1\wedge z^1_2\wedge z^1_3,
\]
where the last equality holds by the Koszul relations.     

If $v_1\dots v_{k-1}\not = v_1'\dots v_{k-1}'$, then the entry in $\beta_{k}\beta_{k+1}'$ is given by 
\[
    \sum_{[v_1\dots v_{k-1}u]=[v_1'\dots v_{k-1}'v]} z^1_u\wedge z^1_{v'}\wedge z^1_{v''}=0.
\]
To see this equality, if $u=v$, the condition $[v_1\dots v_{k-1}u]=[v_1'\dots v_{k-1}'v]$ implies $v_1\dots v_{k-1} = v_1'\dots v_{k-1}'$, which contradicts our assumption. Hence, $u\not=v$ which implies $u\in\{v',v''\}$, and we apply the Koszul relations again to get the desired result. Hence, $\CB_k$ is exact. 
\end{proof}

\begin{remark}
\label{exact C}
It is straightforward to see that the following complexes of $\kk$-modules are exact:
\begin{align*}
  \CC_1\colon & 0 \to A_0^{a_1-3}\xra{[\gamma_1]}C_1\to 0, \notag \\
  \CC_2\colon & 0 \to A_0^{a_2-3}\xra{[\gamma_2]}C_2\to 0, \label{exact C} \\
  \CC_3\colon & 0 \to A_0^{a_3}\xra{[\gamma_3]}C_3\to 0, \notag 
\end{align*}
where $C_j$ is in the homological position $0$ in the complex $\CC_j$, for all $1 \les j \les 3$, and $\gamma_1$ is the $1\times (a_1-3)$ matrix with entries in $K_1$, $\gamma_2$ is the $1\times (a_2-3)$ matrix with entries in $K_2$, and $\gamma_3$ is the $1\times a_3$ matrix with entries in $K_3$ respectively given by:
\begin{equation}
    \label{def gamma}
    \gamma_1=\begin{pmatrix} z^1_4 &\dots& z^1_{a_1} \end{pmatrix}, \qquad \gamma_2=\begin{pmatrix} z^2_1 &\dots& z^2_{a_2-3}\end{pmatrix}, \qquad 
\gamma_3=\begin{pmatrix} z^3_1 &\dots& z^3_{a_3}\end{pmatrix}.
\end{equation}
\end{remark}

\begin{definition}
\label{def seq T}
For any integer $k \ges 1$, let $\alpha_k$ be the $\ell_{k-1}\times \ell_{k}$ matrix with entries in $K_1$, $\alpha_k'$ be the $\ell_{k-1}\times \ell_k'$ matrix with entries in $K_2$, and $\alpha_k''$ be the $\ell_{k-1}\times \ell_k''$ matrix with entries in $K_3$ respectively given by:
\begin{equation}
\label{alpha and alpha'} 
  \alpha_k =\left(\bigoplus_{i=0}^{k-1}\beta_{k-i}^{d_{i}}\Big|\gamma_1^{\ell_{k-1}}\right), \qquad 
  \alpha_k'=\left( \begin{array}{c|c}\begin{array}{c} 0_{d_{k-1} \times \ell_{k-2}} \\ \hline \bigoplus_{i=0}^{k-2}\beta_{k-i}'^{d_{i}} \end{array} & \gamma_2^{\ell_{k-1}} \end{array} \right), \qquad \alpha_k''=\gamma_3^{\ell_{k-1}},
\end{equation}
where the matrices $\beta_i, \beta'_i$ are defined in \Cref{def beta'}, the matrices $\gamma_1,\gamma_2,\gamma_3$ are defined in \eqref{def gamma}, and 
$\{\ell_k\}_{k\ges 0}$, $\{\ell_k'\}_{k\ges 0}$, $\{\ell_k''\}_{k\ges 0}$, and $\{d_k\}_{k\ges 0}$ are sequences of integers satisfying the following recurrence relations:
\begin{align}
\label{ell}
    \ell_k&=\sum_{i=0}^{k}b_{k-i}d_{i},\quad  \text{ where } \\
    d_0&=1 \text{ and }  d_k=(a_1-3)\ell_{k-1},\notag \\
\label{ellprime}
    \ell_0'&=0 \text{ and } \ell_k'=\ell_{k-2}+(a_2-3)\ell_{k-1},\\
\label{elldprime}    
     \ell_0''&=0 \text{ and } \ell_k''=a_3\ell_{k-1}.
\end{align}
\end{definition}

\begin{remark}
\label{ellex}
We list the first few elements of these sequences in terms of $a_i= \rank_{\kk} (A_i)$:
\begin{align*} 
   \ell_0&=1, &\ell_1 &=a_1,         &\ell_2&=a_1^2-3,                   &\ell_3&=a_1^3-6a_1+1,\\
    d_0&=1, &d_1&=a_1-3, &d_2&=(a_1-3)a_1, &d_3&=(a_1-3)(a_1^2-3), \\
   \ell'_0&=0, &\ell_1'&=a_2-3,  &\ell_2'&=a_1a_2-3a_1+1,  &\ell_3'&=a_1^2a_2 - 3a_1^2-3a_2+a_1+9,\\
   \ell''_0&=0, &\ell_1''&=a_3, &\ell_2''&= a_1a_3, &\ell_3''&=a_1^2a_3-3a_3. 
\end{align*}
\end{remark}

\begin{remark}
\label{rmk T}
It is straightforward to check that the sequences of integers $\{b_k\}_{k\ges 0}$ and $\{\ell_k\}_{k\ges 0}$
have respectively the following generating functions:
\begin{equation*}
  \frac{1}{(1-t)^3}\quad\text{and}\quad  \frac{1}{(1-t)^3-t(a_1-3)}.
\end{equation*}
In particular, the sequences of integers
$\{\ell_k'\}_{k\ges 0}$, $\{\ell_k''\}_{k\ges 0}$, and $\{d_k\}_{k\ges 0}$ 
have respectively the following  generating functions:
\begin{equation*}
   \frac{t^2+t(a_2-3)}{(1-t)^3-t(a_1-3)}, \quad \frac{ta_3}{(1-t)^3-t(a_1-3)}, \text{ and } \quad \frac{t(a_1-3)}{(1-t)^3-t(a_1-3)}+1.
\end{equation*} 
\end{remark}

\begin{theorem}
\label{thm:exact A}
Let $(R,\fm, \kk)$ be a local ring of codepth 3 and of class $\clT$, $A$ its Koszul homology,   and retain the above notation. Then for any integer $k\ges 1$, there exists an exact complex of $\kk$-modules:
\begin{equation*}
\CA_k\colon    
 \xymatrixrowsep{2pc}
 \xymatrixcolsep{3pc}
 \xymatrix{
 0\ar@{->}[r]^{} 
 &A_0^{\ell_k}\ar@{->}[r]^{[\alpha_k]}
 &A_1^{\ell_{k-1}} \ar@{->}[r]^{[\alpha_{k-1}]}\ar@{}[d]_{\bigoplus}
 &A_2^{\ell_{k-2}}\ar@{->}[r]^{[\alpha_{k-2}]}\ar@{}[d]_{\bigoplus}
 &A_3^{\ell_{k-3}}\ar@{->}[r]^{}
 &0.
 \\
 &           
 &A_0^{\ell'_{k-1}}\ar@{->}[ru]_{[\alpha'_{k-1}]}          
 &A_0^{\ell''_{k-2}}\ar@{->}[ru]_{[\alpha''_{k-2}]}
 &
 &
 } 
\end{equation*}
with $A_0^{\ell_k}$ is placed in the homological position $k$ in the complex $\CA_k$. The sequences of integers $\ell_k, \ell_k', \ell_k''$ 
and the matrices $\alpha_k, \alpha_k', \alpha_k''$ are given in \Cref{def seq T}.
\end{theorem}

\begin{proof}
Observe that in class $\clT$, for any integer $k\ges 1$ the map $[\alpha_{k-2}]: A_2^{\ell_{k-2}} \to A_3^{\ell_{k-3}}$ is a zero map and the map $[\alpha_{k-2}'']: A_0^{\ell''_{k-2}} \to A_3^{\ell_{k-3}}$ defined in \eqref{alpha and alpha'} is an isomorphism. The complex $\CA_k$ in the theorem is a direct sum of the complexes $\CA'_k$ and $\Sigma^{k-3}\CC_3^{\ell_{k-3}}$, where
\[
\CA_k'\colon    \xymatrixrowsep{1.5pc}
 \xymatrixcolsep{2pc}
 \xymatrix{
 0\ar@{->}[r]^{} 
 &A_0^{\ell_k}\ar@{->}[r]^{[\alpha_k]}
 &A_1^{\ell_{k-1}} \ar@{->}[r]^{[\alpha_{k-1}]}\ar@{}[d]_{\bigoplus}
 &A_2^{\ell_{k-2}}\ar@{->}[r]^{}
 &0,
 \\
 &           
 &A_0^{\ell'_{k-1}}\ar@{->}[ru]_{[\alpha'_{k-1}]}          
 &
 &
 }
\]
with $A_0^{\ell_k}$ in the homological degree $k$. Thus, it suffices to show that $\CA_k'$ is exact for each $k \ges 1$. By construction, we have the following isomorphism of complexes: 
\[ \CA_k' \cong \CB_k^{d_0}\oplus\dots \oplus \Sigma^{k-1}\CB_{1}^{d_{k-1}}\oplus \Sigma^{k-1}\CC_1^{\ell_{k-1}}  \oplus \Sigma^{k-2}\CC_2^{\ell_{k-2}}, \]
where the complexes $\CB_i$ and $\CC_j$ are exact by \Cref{exact B} and \Cref{exact C} respectively. Hence, $\CA_k'$ is exact, and we have an isomorphism of complexes of $\kk$-modules:
\[ \CA_k \cong \left(\bigoplus_{i=0}^{k-1} \Sigma^{i}\CB_{k-i}^{d_{i}}\right) \oplus
\left( \bigoplus_{j=1}^3\Sigma^{k-j}\CC_j^{\ell_{k-j}}\right). \]
\end{proof}

\begin{remark}
In what follows, we illustrate the structure of $\CA_k$ for small cases. Observe that the map $[\gamma_1]$ is a zero map, except when mapping from $A_0^{a_1-3}$, and we use the identification $A_0=B_0 \cong \kk$ throughout. 

\medskip

\underline {Case $k=1$:} By definitions, we have $\ell_0=1,\, \ell_1=a_1$, $\ell'_0=0$, $d_0 = 1$, and $\alpha_1=(\beta_1^{d_0}|\gamma_1^{\ell_0}) = (z^1_1\dots z^1_{a_1})$.  Therefore, it is clear that the sequence \[\CA_1\colon \ 0\to A_0^{a_1}\xra{([z^1])} A_1\to 0\] is exact and there is an isomorphism of complexes $\CA_1\cong \CB_1^{d_0}\oplus \CC_1^{\ell_0}$:
\[
\xymatrixrowsep{1pc}
 \xymatrixcolsep{2pc}
 \xymatrix{
  0\ar@{->}[r]^{} 
 &B_0^{b_1}\ar@{->}[r]^{[\beta_1]}\ar@{}[d]_{\bigoplus}         
 &B_1^{b_0}\ar@{->}[r]   \ar@{}[d]_{\bigoplus}          
 &0
 \\ 
 0\ar@{->}[r]^{} 
 &A_0^{a_1-3}\ar@{->}[r]^{[\gamma_1]}
 &C_1\ar@{->}[r]^{}
 &0.
 }
\]

\medskip

\underline {Case $k=2$:} By using the equalities  
\[
\ell_0=b_0d_0=1, \qquad \ell_1=b_1d_0+b_0d_1=a_1, \qquad \ell_2=b_2d_0+b_1d_1+b_0d_2,\qquad 
\ell_1'= a_2-3, \qquad d_2=(a_1-3)\ell_1, \]
and the maps
\[ 
\alpha_2 = \begin{pmatrix}\beta_2^{d_0}\oplus\beta_1^{d_1}|\gamma_1^{\ell_1}\end{pmatrix}, \qquad \alpha'_1=\gamma_2^{\ell_0},  
\]
we get that 
\[
\CA_2\colon    \xymatrixrowsep{1.5pc}
 \xymatrixcolsep{2pc}
 \xymatrix{
 0\ar@{->}[r]^{} 
 &A_0^{\ell_2}\ar@{->}[r]^{[\alpha_2]}
 &A_1^{\ell_{1}} \ar@{->}[r]^{[\alpha_{1}]}\ar@{}[d]_{\bigoplus}
 &A_2^{\ell_{0}}\ar@{->}[r]^{}
 &0
 \\
 &           
 &A_0^{\ell'_{1}}\ar@{->}[ru]_{[\alpha'_{1}]}          
 &
 &
 }
\]
is isomorphic to the following sum of complexes:
\[
\xymatrixrowsep{1pc}
 \xymatrixcolsep{2pc}
 \xymatrix{ 
 0\ar@{->}[r]^{} 
 &B_0^{b_2d_0}\ar@{->}[r]^{[\beta_2^{d_0}]}\ar@{}[d]_{\bigoplus}
 &B_1^{b_1d_0} \ar@{->}[r]^{[\beta_1^{d_0}]}\ar@{}[d]_{\bigoplus}
 &B_2^{b_0d_0}\ar@{->}[r]^{}\ar@{}[d]_{\bigoplus}
 &0
 \\ 
  0\ar@{->}[r]^{} 
 &B_0^{b_1d_1}\ar@{->}[r]^{[\beta_1^{d_1}]}\ar@{}[d]_{\bigoplus}         
 &B_1^{b_0d_1}\ar@{->}[r]   \ar@{}[d]_{\bigoplus}          
 &0\ar@{->}[r]^{}\ar@{}[d]_{\bigoplus} 
 &0
 \\ 
 0\ar@{->}[r]^{} 
 & A_0^{(a_1-3)\ell_1}\ar@{->}[r]^{[\gamma_1^{\ell_1}]}\ar@{}[d]_{\bigoplus}
 & C_1^{\ell_1}\ar@{->}[r]^{} \ar@{}[d]_{\bigoplus}
 &0\ar@{->}[r]^{} \ar@{}[d]_{\bigoplus}
 &0
 \\
 0\ar@{->}[r]^{} 
 &0\ar@{->}[r]^{} 
 & A_0^{a_2-3}\ar@{->}[r]^{[\gamma_2]}
 & C_2^{}\ar@{->}[r]^{} 
 &0.
 }
\]
That is, there is an isomorphism of complexes $\CA_2\cong\CB_2^{d_0} \oplus \Sigma \CB_1^{d_1}\oplus\Sigma\CC_1^{\ell_1}\oplus \CC_2^{\ell_0}$.

\medskip

\underline {Case $k=3$:} By using the equalities 
\[
\ell_3=b_3d_0+b_2d_1+b_1d_2+b_0d_3,\quad
  \ell_2'= b_0d_0 + (a_2-3)\ell_1,\quad
  \ell_1''=a_3\ell_0,\quad
d_2=(a_1-3)\ell_1, \quad d_3=(a_1-3)\ell_2,
\]
and the maps
\[
\alpha_3=\begin{pmatrix}\beta_3^{d_0}\oplus\beta_2^{ d_1}\oplus\beta_1^{d_2}|\gamma_1^{\ell_2}\end{pmatrix},\quad
  \alpha'_2=(\beta_2'^{d_0}|\gamma_2^{\ell_1}), \quad\text{and}\quad \alpha_1''=\gamma_3^{\ell_0}
\]
we get that 
\[
\CA_3\colon    
 \xymatrixrowsep{2pc}
 \xymatrixcolsep{3pc}
 \xymatrix{
 0\ar@{->}[r]^{} 
 &A_0^{\ell_3}\ar@{->}[r]^{[\alpha_3]}
 &A_1^{\ell_{2}} \ar@{->}[r]^{[\alpha_{2}]}\ar@{}[d]_{\bigoplus}
 &A_2^{\ell_{1}}\ar@{->}[r]^{[\alpha_{1}]}\ar@{}[d]_{\bigoplus}
 &A_3^{\ell_{0}}\ar@{->}[r]^{}
 &0.
 \\
 &           
 &A_0^{\ell'_{2}}\ar@{->}[ru]_{[\alpha'_{2}]}          
 &A_0^{\ell''_{1}}\ar@{->}[ru]_{[\alpha''_{1}]}
 &
 &
 } 
\]
is isomorphic to the following sum of complexes:
 \begin{equation*}
\xymatrixrowsep{1pc}
 \xymatrixcolsep{2.5pc}
 \xymatrix{
  0\ar@{->}[r]^{} 
 &B_0^{b_3d_0}\ar@{->}[r]^{[\beta_3^{d_0}]}
 &B_1^{b_2d_0} \ar@{->}[r]^{[\beta_2^{d_0}]}\ar@{}[d]_{\bigoplus}
 &B_2^{b_1d_0}\ar@{->}[r]^{}
 &0 \ar@{->}[r]^{} 
 &0
 \\
&\ar@{}[d]_{\bigoplus}
&B_0^{b_0d_0}\ar@{->}[ru]_{[\beta_2'^{d_0}]}\ar@{}[d]_{\bigoplus}
&\ar@{}[d]_{\bigoplus}
& &
 \\ 
  0\ar@{->}[r]^{} 
 &B_0^{b_2d_1}\ar@{->}[r]^{[\beta_2^{d_1}]}\ar@{}[d]_{\bigoplus}
 &B_1^{b_1d_1} \ar@{->}[r]^{[\beta_1^{d_1}]}\ar@{}[d]_{\bigoplus}
 &B_2^{b_0d_1}\ar@{->}[r]^{}\ar@{}[d]_{\bigoplus}
 &0 \ar@{->}[r]^{} 
 &0
 \\
  0\ar@{->}[r]^{} 
 &B_0^{b_1d_2}\ar@{->}[r]^{[\beta_1^{d_2}]}\ar@{}[d]_{\bigoplus}         
 &B_1^{b_0d_2}\ar@{->}[r]   \ar@{}[d]_{\bigoplus}          
 &0\ar@{->}[r]^{}\ar@{}[d]_{\bigoplus} 
 &0 \ar@{->}[r]^{} 
 &0
 \\ 
 0\ar@{->}[r]^{} 
 & A_0^{(a_1-3)\ell_2}\ar@{->}[r]^{[\gamma_1^{\ell_2}]}\ar@{}[d]_{\bigoplus}
 & C_1^{\ell_2}\ar@{->}[r]^{} \ar@{}[d]_{\bigoplus}
 &0\ar@{->}[r]^{} \ar@{}[d]_{\bigoplus}
 &0 \ar@{->}[r]^{} 
 &0
 \\
 0\ar@{->}[r]^{} 
 &0\ar@{->}[r]^{}  \ar@{}[d]_{\bigoplus}
 & A_0^{(a_2-3)\ell_1}\ar@{->}[r]^{[\gamma_2^{\ell_1}]} \ar@{}[d]_{\bigoplus}
 & C_2^{\ell_1}\ar@{->}[r]^{}  \ar@{}[d]_{\bigoplus}
 &0 \ar@{->}[r]^{}  \ar@{}[d]_{\bigoplus}
 &0
 \\
 0\ar@{->}[r]^{} 
 &0\ar@{->}[r]^{}
 &0\ar@{->}[r]^{} 
 &A_0^{a_3\ell_0} \ar@{->}[r]^{[\gamma_3^{\ell_0}]}
 &C_3^{\ell_0} \ar@{->}[r]^{} 
 &0
 }
\end{equation*}
that is, we get an isomorphism of complexes $\CA_3\cong \CB_3^{d_0}\oplus   \Sigma\CB_2^{d_1}\oplus \Sigma^2\CB_1^{d_2}\oplus\Sigma^2\CC_1^{\ell_2}\oplus \Sigma\CC_2^{\ell_1}\oplus \CC_3^{\ell_0}$.
\end{remark}

\begin{corollary}
\label{cor:acyclic A}
Let $(R,\fm, \kk)$ be a local ring of codepth 3 and of class $\clT$, $A$ its Koszul homology, and retain the above notation. Consider the complex $\BA \coloneqq \bigoplus_{k \ges 0} \CA_k$ given as the sum of complexes of $\kk$-modules, where $\CA_0$ is the complex $0\to A_0^{\ell_0} \to 0$, and $\CA_k$ is as in \Cref{thm:exact A}. 

Then, $\BA$ is an acyclic complex of graded $A$-modules with $\HH{0}{\BA} \cong A_0 \cong \kk$, of the following form:
\begin{equation} 
\label{BA}
\BA \colon \hspace{-0.2in}
 \xymatrixrowsep{3pc}
 \xymatrixcolsep{2.2pc}
 \xymatrix{
 &\dots\ar@{->}[r]^{}
 &A^{\ell_{i}}(-i) \ar@{->}[r]^{[\alpha_{i}]}\ar@{}[d]_{\bigoplus}
 &A^{\ell_{i-1}}(-i+1)\ar@{->}[r]^{[\alpha_{i-1}]}\ar@{}[d]_{\bigoplus}
 &\qquad \dots \qquad \ar@{->}[r]^{[\alpha_{2}]}
 &A^{\ell_{1}}(-1)\ar@{->}[r]^{[\alpha_{1}]}\ar@{}[d]_{\bigoplus}
 &A^{\ell_0}
 \\  
 &\dots\ar@{->}[ru]^{}       
 &\kk^{\ell'_{i}}(-i-1)\ar@{->}[ru]|{[\alpha'_{i}]}\ar@{}[d]_{\bigoplus}
 &\kk^{\ell'_{i-1}}(-i)\ar@{->}[ru]|{[\alpha'_{i-1}]}\ar@{}[d]_{\bigoplus}
 &\dots\ar@{->}[ru]|{[\alpha'_2]}
 &\kk^{\ell'_{1}}(-2)\ar@{->}[ru]|{[\alpha'_{1}]}\ar@{}[d]_{\bigoplus}
 &
  \\   
 &\dots\ar@{->}[ruu]^{}        
 &\kk^{\ell''_{i}}(-i-2)\ar@{->}[ruu]|{[\alpha''_{i}]}
 &\kk^{\ell''_{i-1}}(-i-1)\ar@{->}[ruu]|{[\alpha''_{i-1}]} 
 &\dots\ar@{->}[ruu]|{[\alpha''_2]}
 &\kk^{\ell''_{1}}(-3)\ar@{->}[ruu]|{[\alpha''_{1}]}
 &
 } 
\end{equation}
where the sequences of integers $\ell_k, \ell_k', \ell_k''$ 
and the matrices $\alpha_k, \alpha_k', \alpha_k''$ are given in \Cref{def seq T}.
\end{corollary}

\begin{proof}
By \Cref{thm:exact A}, for each $k \ges 1$, the complex of $\kk$-modules  $\CA_k$ is exact so $\BA = \bigoplus_{k \ges 0} \CA_k$ is an acyclic complex of $\kk$-modules with $\HH{0}{\BA}=A_0^{\ell_0}=A_0\cong \kk$. 

We now show that $\BA$ has the form as in \eqref{BA} as a complex of $A$-modules. Observe that for any $i \ges 0$, the $i$-th homological component of $\CA_k$ can be described as follows,
\begin{equation*}
    (\CA_k)_i=
    \begin{cases}
    A_0^{\ell_{k}},&\text{if } i=k\\[0.05in]
    A_1^{\ell_{k-1}}\oplus  A_0^{\ell'_{k-1}},&\text{if } i=k-1\\[0.05in]
    A_2^{\ell_{k-2}}\oplus  A_0^{\ell''_{k-2}} ,&\text{if } i=k-2\\[0.05in]
    A_3^{\ell_{k-3}},&\text{if } i=k-3\\[0.05in]
    0, &\text{otherwise}.
    \end{cases}
\end{equation*}
Thus, each $i$-th homological component of $\BA$ is given by:
\begin{align*}
    \BA_i&=(\CA_i)_i\oplus (\CA_{i+1})_i\oplus (\CA_{i+2})_i\oplus(\CA_{i+3})_i\\
        &=A_0^{\ell_{i}}\oplus \big( A_1^{\ell_{i}}\oplus  A_0^{\ell'_{i}}\big)\oplus \big( A_2^{\ell_{i}}\oplus  A_0^{\ell''_{i}}\big)\oplus A_3^{\ell_{i}}\\
        &=A^{\ell_{i}}\oplus \kk^{\ell'_{i}}\oplus \kk^{\ell''_{i}}.
\end{align*}
By \Cref{def seq T}, for each $i \ges 1$ the matrices $[\alpha_i], [\alpha_i'], [\alpha_i'']$ have entries in $A_1$, $A_2$, and $A_3$ respectively. This explains the internal shifts in each homological component of $\BA$ in \eqref{BA}. Therefore, $\BA$ becomes an acyclic complex of graded $A$-modules.
\end{proof}

%%%%%%%%%%%%%%%%%%%%%%%%%%%%%%%
\subsection{Constructing the sequences of complexes $\sC^{(k)}$ and of chain maps $\varphi^{(k)}$}
\label{Candphi}

In this subsection, we describe the sequence of complexes $\{\sC^{(k)}\}_{k \ges 0}$ in terms of the Koszul complex $K$, and construct the sequence of chain maps $\{\varphi^{(k)} \colon \sC^{(k)} \to \sC^{(k-1)}\}_{k \ges 1}$ that satisfy the hypotheses of \Cref{graded res} for class $\clT$. To better describe the graded resolution of $\kk$ over $A$ coming from our construction, we need to relabel the sequences of integers and matrices defined in \Cref{def seq T}. 

For all integers $k,r \ges 0$, we set
\begin{align}
\label{ellks}
\ell_{k,r} &\coloneqq \begin{cases}
\ell_k,&\text{ if } r=k\\
\ell_k',&\text{ if } r=k+1\\
\ell_k'',&\text{ if } r=k+2\\
0, &\text{ otherwise}
\end{cases}
\end{align}
and for all integers $k,r \ges 1$, we set
\begin{align}
\label{alphaks}
\alpha_{k,r}&\coloneqq \begin{cases}
\alpha_k,&\text{ if } r=k\\
\alpha_k',&\text{ if } r=k+1\\
\alpha_k'',&\text{ if } r=k+2\\
0, &\text{ otherwise}
\end{cases}\\
\label{def delta}
\Delta_k &\coloneqq 
\begin{pmatrix} \alpha_{k,k} &  \alpha_{k,k+1} & \alpha_{k,k+2}\end{pmatrix}, 
\end{align}
where the sequences of integers $\ell_k, \ell_k', \ell_k''$ 
and the matrices $\alpha_k, \alpha_k', \alpha_k''$ with entries in $K_1,K_2,K_3$, respectively, are given in \Cref{def seq T}.

\begin{lemma} 
\label{alpha chain}
Retain the above notation. Then, for all $k,r,m\ges 1$ we have 
\begin{itemize}
\item[(1)] The multiplication map $\Sigma^{1+r-k} K^{\ell_{k,r}} \xra{\alpha_{k,r}} K^{\ell_{k-1,k-1}}$ is a chain map of complexes. 
\item[(2)] The maps
\(
    (\Sigma K^{\ell_{k,k}} \oplus \Sigma^2 K^{\ell_{k,k+1}} \oplus \Sigma^3 K^{\ell_{k,k+2}})^m \xra{\Delta_k^m} (K^{\ell_{k-1,k-1}})^m,
\)
given by blocks of maps where each $\Delta_k$ appears in the main diagonal of $\Delta_k^m$ and $0$ everywhere else, are chain maps of complexes.
\end{itemize}
\end{lemma}

\begin{proof}
(1): It is enough to show that for all $i,k\ges 1$ and $r\in\{k,k+1,k+2\}$, the diagram below is commutative:
 \begin{equation*}
  \xymatrixrowsep{2pc}
  \xymatrixcolsep{2pc}
  \xymatrix
  {
  K_{i+k-r-1}^{\ell_{k,r}} \ar@{->}[d]_{  (-1)^{r-k+1}\partial_{i+k-r-1}^{\ell_{k,r}} }
                         \ar@{->}[r]^{\alpha_{k,r}}
  & K_{i}^{ \ell_{k-1,k-1} }
  \ar@{->}[d]^{ \partial_{i}^{\ell_{k-1,k-1}} }
  \\
  K_{i+k-r-2}^{\ell_{k,r}} \ar@{->}[r]_{\alpha_{k,r}}
  &  K_{i-1}^{\ell_{k-1,k-1}}.
 }
\end{equation*}
Since all the entries of $\alpha_{k,r}$ are cycles of degree $r-k+1$, the commutativity of the diagram  follows from a more general statement as follows.

Let $\theta$ be an $u\times v$ matrix with entries that are cycles of the Koszul complex $(K,\partial)$ of degree $j\ges 0$. Then, for all $i\ges 0$ the following diagram is commutative:
\begin{equation*}
  \xymatrixrowsep{1.5pc}
  \xymatrixcolsep{1.5pc}
  \xymatrix
  {
  K_{i-j}^{v}\ar@{->}[d]_{  (-1)^{j}\partial_{i-j}^{v}}
                         \ar@{->}[r]^{\theta}
  & K_{i}^{u}
  \ar@{->}[d]^{ \partial_{i}^{u} }
  \\
  K_{i-j-1}^{v} \ar@{->}[r]_{\theta}
  &  K_{i-1}^{u}.
 }
\end{equation*}
Indeed, for each  $(y_k)_{1\les k \les v}\in K_{i-j}^v$ we have: 
\begin{align*}
    (\partial_i^u\circ \theta) (y_k)_{1\les k \les v}
    &=\partial_i^u\left( \sum_{k=1}^v\theta(s,k)\wedge y_k\right)_{1\les s \les u}\\
    &=\left( \sum_{k=1}^v\partial_i(\theta(s,k)\wedge  y_k)\right)_{1\les s \les u}\\
    &=\left( \sum_{k=1}^v\partial_j(\theta(s,k))\wedge y_k+(-1)^j\theta(s,k)\wedge \partial_{i-j}(y_k)\right)_{1\les s \les u}\\
    &=(-1)^j\left(\sum_{k=1}^v\theta(s,k)\wedge \partial_{i-j}(y_k)\right)_{1\les s \les u}\\
    &=(-1)^j\theta\left(\partial_{i-j}(y_k)\right)_{1\les k \les v}\\
    &=\left(\theta\circ (-1)^j\partial_{i-j}^v\right)(y_k)_{1\les k \les v}.
\end{align*}
Finally, (2) is a direct consequence of (1). 
\end{proof}

To facilitate the construction of complexes $\sC^{(k)}$ and maps $\varphi^{(k)}$, we now introduce a bookkeeping tool which also gives a visual illustration of how the resolution $\BF$ of $\kk$ over $A$ grows for class $\clT$. We will see in the proof of \Cref{res k over A class T} that the acyclic complex $\BA$ in \eqref{BA} plays a role of the ``seed" (cf. \eqref{seed}) which builds the tree in \Cref{fig:tree}. 

Consider a non-commutative algebra $\TF$ generated over $\BZ$ by indeterminates $X_{i,i}, X_{i,i+1}, X_{i,i+2}$, all of degree $i \ges 1$ which is recorded by the first index, subject to the relations
\begin{equation}
\label{algebraF}
    \TF \coloneqq \frac{\BZ \langle X_{i,i}, X_{i,i+1}, X_{i,i+2} \colon \forall i \ges 1 \rangle}{\left(X_{i,i}X_{j,j}, \ X_{i,i+1}X_{j,j}, \  X_{i,i+2}X_{j,j}\colon \forall i,j \ges 1 \right)}. 
\end{equation}
Remark that for $k \ges 0$, each graded component $\TF_k$ is a free $\BZ$-module generated by monomials whose total degree is $k$. One can show that $\rank_{\BZ} \TF_k = 3^k$.

Set $X_{0,0}=1$ and $X_{0,1}=0=X_{0,2}$. We associate to $\TF$  the nodes of a  directed rooted tree. The first portion of the tree, that  corresponds to monomials of total degrees $\les 1$, is given by:

\begin{equation}
\label{seed}
 \xymatrixrowsep{1pc}
 \xymatrixcolsep{3.5pc}
 \xymatrix{
\dots \ar@{->}[r]^{e_{4,4}}
&X_{3,3}\ar@{->}[r]^{e_{3,3}} 
&X_{2,2}\ar@{->}[r]^{e_{2,2}} 
&X_{1,1}\ar@{->}[r]^{e_{1,1}}
&{1}.
\\
\dots \ar@{->}[ur]|{e_{4,5}}
&X_{3,4}\ar@{->}[ur]|{e_{3,4}}
&X_{2,3}\ar@{->}[ur]|{e_{2,3}} 
&X_{1,2}\ar@{->}[ur]|{e_{1,2}}
&
\\
\dots \ar@{->}[uur]|{e_{4,6}}
&X_{3,5}\ar@{->}[uur]|{e_{3,5}} 
&X_{2,4}\ar@{->}[uur]|{e_{2,4}} 
&X_{1,3}\ar@{->}[uur]|{e_{1,3}}
&
}
\end{equation}
We reefer to it as the {\it seed} of the tree.

The node set of this tree consists of monomials in $\TF$. We denote by $\bn$ a monomial only given in terms of $X_{i,i+1}$ and $X_{j,j+2}$ for some $i,j \ges 0$; note that $\bn$ could be $1$. From the relations in $\TF$, every monomial (node) is either of the form $X_{i,i} \cdot \bn$ or $\bn$. 
The arrows of the tree go only from total degree $k$ monomials to total degree $k-1$ monomials in the following way. For any $i \ges 1$, there are three arrows going to each node $X_{i-1,i-1} \cdot \bn$ of the form:
\begin{equation}
\label{arrow}
X_{i,r}\cdot \bn \xra{e_{i,r}(\bn)} X_{i-1,i-1}\cdot \bn,\quad\text{where}\ r \in \{i, i+1, i+2\}. 
\end{equation}

In \Cref{fig:tree}, we display the first few layers of this tree, going into the root $1$, to demonstrate how the tree grows. Note that each $k$-th vertical layer contains monomials, arranged in specific order, whose total degree (sum of the first indices in each monomial) is $k$ and generate $\TF_k$. 
\begin{figure}[h!]   
\centering
    \caption{A directed rooted tree associated to the class $\clT$}
    \label{fig:tree}
\begin{equation*}
 \xymatrixrowsep{1.5pc}
 \xymatrixcolsep{8pc}
\SMALL{ \xymatrix{ 
\dots X_{3,3}\ar@{->}[r]^{e_{3,3}} 
&X_{2,2}\ar@{->}[r]^{e_{2,2}} 
&X_{1,1}\ar@{->}[r]^{e_{1,1}}
&{1}
\\
X_{3,4}\ar@{->}[ur]|{e_{3,4}}
&X_{2,3}\ar@{->}[ur]|{e_{2,3}} 
&X_{1,2}\ar@{->}[ur]|{e_{1,2}}
&
\\
X_{3,5}\ar@{->}[uur]|{e_{3,5}} 
&X_{2,4}\ar@{->}[uur]|{e_{2,4}} 
&X_{1,3}\ar@{->}[uur]|{e_{1,3}}
&
\\
X_{1,1}X_{2,3}\ar@{->}[uur]|{e_{1,1}(X_{2,3})}
&X_{1,1}X_{1,2}\ar@{->}[uur]|{e_{1,1}(X_{1,2})}
&
&
\\
X_{1,2}X_{2,3}\ar@{->}[uuur]|{e_{1,2}(X_{2,3})}
&X_{1,2}X_{1,2}\ar@{->}[uuur]|{e_{1,2}(X_{1,2})}
&
&
\\
X_{1,3}X_{2,3}\ar@{->}[uuuur]|{e_{1,3}(X_{2,3})}
&X_{1,3}X_{1,2}\ar@{->}[uuuur]|{e_{1,3}(X_{1,2})}
&
&
\\
X_{1,1}X_{2,4}\ar@{->}[uuuur]|{e_{1,1}(X_{2,4})}
&X_{1,1}X_{1,3}\ar@{->}[uuuur]|{e_{1,1}(X_{1,3})}
&
&
\\
X_{1,2}X_{2,4}\ar@{->}[uuuuur]|{e_{1,2}(X_{2,4})}
&X_{1,2}X_{1,3}\ar@{->}[uuuuur]|{e_{1,2}(X_{1,3})}
&
&
\\
X_{1,3}X_{2,4}\ar@{->}[uuuuuur]|{e_{1,3}(X_{2,4})}
&X_{1,3}X_{1,3}\ar@{->}[uuuuuur]|{e_{1,3}(X_{1,3})}
&
&
\\
X_{2,2}X_{1,2} \ar@{->}[uuuuuur]|{e_{2,2}(X_{1,2})}
& 
& 
& 
\\
X_{2,3}X_{1,2} \ar@{->}[uuuuuuur]|{e_{2,3}(X_{1,2})}
& 
& 
& 
\\
X_{2,4}X_{1,2}  \ar@{->}[uuuuuuuur]|{e_{2,4}(X_{1,2})}
& 
& 
& 
\\
X_{1,1}X_{1,2}X_{1,2} \ar@{->}[uuuuuuuur]|{e_{1,1}(X_{1,2}X_{1,2})}
& 
& 
& 
\\
X_{1,2}X_{1,2}X_{1,2} \ar@{->}[uuuuuuuuur]|{e_{1,2}(X_{1,2}X_{1,2})}
& 
& 
& 
\\
X_{1,3}X_{1,2}X_{1,2}  \ar@{->}[uuuuuuuuuur]|{e_{1,3}(X_{1,2}X_{1,2})}
& 
& 
& 
\\
\vdots 
& 
& 
& 
}}
\end{equation*}
\end{figure}

\medskip

Using this tree, we now construct the sequence of complexes $\{\sC^{(k)}\}_{k \ges 0}$ and the sequence of chain maps $\{\varphi^{(k)} \colon \sC^{(k)} \to \sC^{(k-1)}\}_{k \ges 1}$ that satisfy the hypotheses of the main \Cref{graded res} for class $\clT$.

First, for each $k \ges 0$, we associate each $k$-th vertical layer in the above tree to the complex $\sC^{(k)}$ that is a direct sum of copies of the Koszul complex $K$ with the corresponding shifts, as follows. Each monomial in $\TF$ has the form:
\begin{equation*}
\bm=\prod_{i=1}^j X_{k_i,s_i},\ \text{where } 1\les k_1\les s_1\les k_1+2\text{ and } 1 \les k_i< s_i\les k_i+2\ \text{for all } i\ges 2.
\end{equation*}
To each monomial $\bm$ above, we associate four degrees:
\begin{equation}
\label{deg}
\deg_1\bm\coloneqq\sum_{i=1}^jk_i,\qquad \deg_2\bm\coloneqq\sum_{i=1}^j s_i,\qquad \deg_3\bm\coloneqq\prod_{i=1}^j\ell_{k_i,s_i},\qquad\text{ and }\qquad \deg_4\bm\coloneqq j.
\end{equation} 
The component of $\sC^{(k)}$, where $k=\deg_1\bm$ that corresponds to the monomial $\bm$, is given by:
\begin{equation*}
%\label{identification}
\bm \mapsto \Sigma^s K^{u},\quad \text{where } s=\deg_2\bm\text{ and } u=\deg_3\bm.
\end{equation*} 
Here, the first degree of $\bm$ is the total monomial degree in $\TF$, the second degree associates to the shifting or internal grading of $K$, the third degree associates to the number of direct sums of copies of $K$, and the fourth degree is the number of indeterminates appearing in $\bm$. 
In particular, we identify each indeterminate in $\TF$ with a shifting of copies of $K$ as $X_{k,s}\mapsto \Sigma^{s}K^{\ell_{k,s}}$, for all $k\ges 0$ and $s\in\{k,k+1,k+2\}$.

In conclusion, for all $k\ges 0$,  we set 
\begin{equation}
\label{def Ck}
    \sC^{(k)}\coloneqq \bigoplus_{\deg_1\bm=k}  \Sigma^{\deg_2\bm} K^{\deg_3\bm}   
\end{equation}
where $\bm$ is a monomial of $\TF$ (cf.~\eqref{algebraF}) and different degrees of $\bm$ are defined in \eqref{deg}. Observe that $\sC^{(0)}=K$ as $\bm=1$ of first degree $0$.

From the description of $\bm$ and its degrees, it is straightforward to obtain the following result.
\begin{lemma} 
\label{lem Ck}
For all $k\ges 0$, we have an isomorphism of complexes: \(\sC^{(k)}\cong \bigoplus_{s\ges 0} \Sigma^s K^{u_{k,s}}\),
where 
\begin{equation}
\label{uks}
u_{k,s} \coloneqq \sum_{\substack{k_1+\dots+k_j=k\\ s_1+\dots+s_j=s\\  s_i\neq k_i, \forall i\ges 2}} \ell_{k_1,s_1}\ell_{k_2,s_2}\dots\ell_{k_j,s_j}, \quad \text{ for each $k,s \ges 0$}.
\end{equation}
\end{lemma}

For example, the $0$-th vertical layer $\TF_0=1$ , the $1$-st vertical layer $\TF_1 = \Span_\kk \{X_{1,1},X_{1,2},X_{1,3}\}$, and the $2$-nd vertical layer $\TF_2 = \Span_\kk \{X_{2,2},X_{2,3},X_{2,4}, X_{1,1}X_{1,2}, X_{1,2}X_{1,2}, X_{1,3}X_{1,2}, X_{1,1}X_{1,3}, X_{1,2}X_{1,3}, X_{1,3}X_{1,3}\}$ correspond respectively to the following complexes: 
\begin{align*}
\sC^{(0)}&\coloneqq K^{\ell_{0,0}}, \\
\sC^{(1)}&\coloneqq \Sigma K^{\ell_{1,1}}\oplus\Sigma^2 K^{ \ell_{1,2}}\oplus \Sigma^3K^{\ell_{1,3}},\\
\sC^{(2)}&\coloneqq\Sigma^2 K^{\ell_{2,2}}\oplus\Sigma^3 K^{\ell_{2,3}}\oplus \Sigma^4K^{\ell_{2,4}}\\
    &\oplus\left(\Sigma^3 K^{\ell_{1,1}}\oplus\Sigma^4 K^{\ell_{1,2}}\oplus \Sigma^5K^{\ell_{1,3}}\right)^{\ell_{1,2}}
     \oplus\left(\Sigma^4 K^{\ell_{1,1}}\oplus\Sigma^5 K^{\ell_{1,2}}\oplus \Sigma^6K^{\ell_{1,3}}\right)^{\ell_{1,3}}.
\end{align*}
Similarly, we get the following:
\begin{align*}
\sC^{(3)}&\coloneqq\Sigma^3 K^{\ell_{3,3}}\oplus\Sigma^4 K^{\ell_{3,4}}\oplus \Sigma^5K^{\ell_{3,5}}\\
     &\oplus\left(\Sigma^4 K^{\ell_{1,1}}\oplus\Sigma^5 K^{\ell_{1,2}}\oplus \Sigma^6K^{\ell_{1,3}}\right)^{\ell_{2,3}}
    \oplus\left(\Sigma^5 K^{\ell_{1,1}}\oplus\Sigma^6 K^{\ell_{1,2}}\oplus \Sigma^7 K^{\ell_{1,3}}\right)^{\ell_{2,4}}\\
    &\oplus\left(\Sigma^4 K^{\ell_{2,2}}\oplus\Sigma^5 K^{\ell_{2,3}}\oplus \Sigma^6 K^{\ell_{2,4}}\right)^{\ell_{1,2}}\\
    &\oplus\left(\Sigma^5 K^{\ell_{1,1}}\oplus\Sigma^6 K^{\ell_{1,2}}\oplus \Sigma^7 K^{\ell_{1,3}}\right)^{(\ell_{1,2})^2}
      \oplus\left(\Sigma^6 K^{\ell_{1,1}}\oplus\Sigma^7 K^{\ell_{1,2}}\oplus \Sigma^8 K^{\ell_{1,3}}\right)^{\ell_{1,3}\ell_{1,2}}\\
       &\oplus\left(\Sigma^5 K^{\ell_{2,2}}\oplus\Sigma^6 K^{\ell_{2,3}}\oplus \Sigma^7 K^{\ell_{2,4}}\right)^{\ell_{1,3}}\\
       &\oplus\left(\Sigma^6 K^{\ell_{1,1}}\oplus\Sigma^7 K^{\ell_{1,2}}\oplus \Sigma^8 K^{\ell_{1,3}}\right)^{\ell_{1,2}\ell_{1,3}}
        \oplus\left(\Sigma^7 K^{\ell_{1,1}}\oplus\Sigma^8 K^{\ell_{1,2}}\oplus \Sigma^9 K^{\ell_{1,3}}\right)^{(\ell_{1,3})^2}.
\end{align*}
Observe that given this arrangement, each $\sC^{(k)}$ has $3^k$ terms, with each term corresponds to a monomial of first degree $k$ in $\TF$.

\medskip

Second, for each $k\ges 1$, using the above arrangements of $\sC^{(k)}$, we construct a map $\varphi^{(k)}\colon \sC^{(k)} \to \sC^{(k-1)}$. To each group of three arrows defined in \eqref{arrow}, we associate a map 
\[
    \begin{pmatrix} e_{i,i}(\bn) &  e_{i,i+1}(\bn) & e_{i,i+2}(\bn) \end{pmatrix} \mapsto (\Delta_i)^{\deg_3 \bn}, \text{ for all } i \ges 1,
\]
where $\bn$ is a monomial only given in terms of $X_{i,i+1}$ and $X_{j,j+2}$ for some $i,j \ges 0$ and $\Delta_i$ is defined in \eqref{def delta}. We define each chain map $\varphi^{(k)}$ recursively as blocks of maps: 
\begin{equation}
\label{matrix phi}
    \varphi^{(1)} \coloneqq \Delta_1 \qquad \text{ and } \qquad 
\varphi^{(k+1)} \coloneqq\begin{pmatrix}
\Delta_{k+1}&0&0&0&0\\
0&\Delta_1^{\ell_{k,k+1}}&0&0&0\\
0&0&\Delta_1^{\ell_{k,k+2}}&0&0\\
0&0&0&(\varphi^{(k)})^{\ell_{1,2}}&0\\
0&0&0&0&(\varphi^{(k)})^{\ell_{1,3}}\\
\end{pmatrix}.
\end{equation}

For example, we have the following 
\begin{align*}
\varphi^{(2)}&=\begin{pmatrix} 
\Delta_2&0&0\\
0&\Delta_1^{\ell_{1,2}}&0\\
0&0&\Delta_1^{\ell_{1,3}}
\end{pmatrix} 
\quad 
\text{and}\\[0.05in]
\varphi^{(3)}&=\begin{pmatrix} 
\Delta_3&0&0&0&0&0&0&0&0\\
0&\Delta_1^{\ell_{2,3}}&0&0&0&0&0&0&0\\
0&0&\Delta_1^{\ell_{2,4}}&0&0&0&0&0&0\\
0&0&0&\Delta_2^{\ell_{1,2}}&0&0&0&0&0\\
0&0&0&0&\Delta_1^{(\ell_{1,2})^2}&0&0&0&0\\
0&0&0&0&0&\Delta_1^{\ell_{1,3}\ell_{1,2}}&0&0&0\\
0&0&0&0&0&0&\Delta_2^{\ell_{1,3}}&0&0\\
0&0&0&0&0&0&0&\Delta_1^{\ell_{1,2}\ell_{1,3}}&0\\
0&0&0&0&0&0&0&0&\Delta_1^{(\ell_{1,3})^2}
\end{pmatrix}.
\end{align*}

%%%%%%%%%%%%%%%%%%%%%%%%%%%%%%%
\subsection{A graded minimal resolution of $\kk$ over $A$}
\label{graded res over A}

In this subsection, we show that the sequences of complexes $\sC^{(k)}$ and of maps $\varphi^{(k)}$ constructed in \Cref{Candphi} indeed give us the setting for \Cref{graded res} and give rise to a graded minimal resolution of $\kk$ over $A$ for class $\clT$.

\begin{theorem}
\label{res k over A class T}
Let $(R,\fm, \kk)$ be a local ring of codepth 3 and of class $\clT$, $A$ its Koszul homology which we consider as a graded-commutative ring, and retain the notation in \Cref{Candphi}. For all $k \ges 1$, set $\partial^\BF_k = [\varphi^{(k)}]$. Then, a graded minimal resolution of the residue field $\kk$ over $A$  has the form:
\begin{equation*}
\label{curlyF}
    \BF\colon \dots\to \bigoplus_{s=k}^{3k} A^{u_{k,s}}(-s)\xra{\partial_k^\BF} \bigoplus_{s=k-1}^{3(k-1)} A^{u_{k-1,s}}(-s)\xra{\partial_{k-1}^\BF}\dots \to \bigoplus_{s=1}^{3} A^{u_{1,s}}(-s) \xra{\partial_{1}^\BF} A,
\end{equation*}
where $u_{k,s}$ is defined in \eqref{uks} for all $k,s$.
\end{theorem}

\begin{proof} 
It follows immediately from \Cref{alpha chain} that for each $k\ges 1$, the map $\varphi^{(k)}\colon \sC^{(k)} \to \sC^{(k-1)}$ is a chain map. Thus, $[\varphi^{(k)}]\colon \HH{}{\sC^{(k)}}\to \HH{}{\sC^{(k-1)}}$ is well-defined. By \Cref{lem Ck}, using that the homological shift of $K$ in $\sC^{(k)}$ becomes the internal grading shift of $A$ as a graded $\kk$-algebra, we get \(\BF_k\coloneqq \bigoplus_{s=k}^{3k} A^{u_{k,s}}(-s)\cong \HH{}{\sC^{(k)}}\), thus 
$\partial^\BF_k$ is well-defined.

To show that $\BF$ described in the statement is an acyclic complex, we give it a different description. Let $\varepsilon\colon \BA\to \kk$ be a quasi-isomorphism obtained from \Cref{cor:acyclic A} with $\varepsilon_i=0$ for all $i>0$ and $\varepsilon_0=\id{\kk}$. In \Cref{Candphi}, we associate the components of $\sC^{(k)}$ and $\varphi^{(k)}$ to the nodes (monomials in $\TF$) and arrows of the directed tree in \Cref{fig:tree} respectively. One can see that $\BF$ can be obtained from the complex $\BA$ given in \eqref{BA}, which corresponds to the seed \eqref{seed} of the tree, by iteratively replacing $\kk$ in \eqref{BA} by the complex $\BA$ via the quasi-isomorphism $\varepsilon$. More precisely, we obtain a sequence of quasi-isomorphisms
\[ \varepsilon^{(j+1)}\colon \BA^{(j+1)}\xra{\simeq} \BA^{(j)}, \qquad\text{for all }j\ges 0,  \]
where $\varepsilon^{(1)}\coloneqq\varepsilon$, $\BA^{(0)}\coloneqq \kk$, $\BA^{(1)} \coloneqq \BA$, and for example, $\BA^{(2)}$ is given by
\begin{equation*} 
 \xymatrixrowsep{3.3pc}
 \xymatrixcolsep{3.7pc}
 \xymatrix{
 \dots\ar@{->}[r]^{}
 &A^{\ell_{i}}(-i) \ar@{->}[r]^{[\alpha_{i}]}\ar@{}[d]_{\bigoplus}
 &A^{\ell_{i-1}}(-i+1)\ar@{->}[r]^{[\alpha_{i-1}]}\ar@{}[d]_{\bigoplus}
 &\dots
 &A^{\ell_{1}}(-1)\ar@{->}[r]^{[\alpha_{1}]}\ar@{}[d]_{\bigoplus}
 &A^{\ell_0}.
 \\  
 \dots\ar@{->}[ru]^{}       
 &\BA^{\ell'_{i}}(-i-1)\ar@{->}[ru]|{[\alpha'_{i}\circ \varepsilon^{\ell_i'}]}\ar@{}[d]_{\bigoplus}
 &\BA^{\ell'_{i-1}}(-i)\ar@{->}[ru]|{[\alpha'_{i-1}\circ \varepsilon^{\ell_{i-1}'}]}\ar@{}[d]_{\bigoplus}
 &\dots 
 &\BA^{\ell'_{1}}(-2)\ar@{->}[ru]|{[\alpha'_{1} \circ\varepsilon^{\ell_1'}]}\ar@{}[d]_{\bigoplus}
 &
  \\   
 \dots \ar@{->}[ruu]^{}        
 &\BA^{\ell''_{i}}(-i-2)\ar@{->}[ruu]|{[\alpha''_{i}\circ \varepsilon^{\ell_i''}]}
 &\BA^{\ell''_{i-1}}(-i-1)\ar@{->}[ruu]|{[\alpha''_{i-1}\circ \varepsilon^{\ell_{i-1}''}]} 
 &\dots
 &\BA^{\ell''_{1}}(-3)\ar@{->}[ruu]|{[\alpha''_{1}\circ\varepsilon^{\ell_1''}]}
 &
 } 
\end{equation*}
The $k$-th homological component of each $\BA^{(j)}$ is given by:
\begin{equation*}
    \BA^{(j)}_k \coloneqq  \left(\bigoplus_{\substack{\deg_1\bm = k\\ \deg_4\bm =j}} A^{\deg_3\bm}(-\deg_2\bm) \right) \oplus \left(\bigoplus_{\substack{\deg_1\bn = k\\ \deg_4\bn =j}} \kk^{\deg_3\bn}(-\deg_2\bn)\right),
\end{equation*}
where $\bm$ is a monomial in $\TF$ that contains $X_{i,i}$ and $\bn$ is a monomial in $\TF$ that does not contain $X_{i,i}$. 
In fact, for all $k \ges 1$ we have  $\BF_{\les k} = \BA_{\les k}^{(k+1)}$ by reading the vertical components. 
First of all, this shows that $\BF$ is a complex of graded $A$-modules. Second, for all $i\ges 1$, by construction via the quasi-isomorphisms, we have:
\[\HH{i}{\BF}= \HH{i}{\BA^{(k+1)}}\cong\HH{i}{\BA}\quad\text{for all}\ k \ges i+1.\]   
 By \Cref{cor:acyclic A}, $\BA$ is an acyclic complex of graded $A$-modules, so $\BF$ is also acyclic.
 The minimality of $\BF$ follows since the differential maps $\partial_k^\BF$ all have entries in $A_{\ges 1}$. Therefore, $\BF$ is a graded minimal resolution of $\kk$ over $A$ of the desired format.
\end{proof}

\begin{remark} 
The first few components of $\BF$ are: 
\begin{align*}
\BF_0&\cong A^{\ell_{0,0}} \\
\BF_1&\cong A^{\ell_{1,1}}(-1)\oplus A^{\ell_{1,2}}(-2)\oplus A^{\ell_{1,3}}(-3)\\
\BF_2&\cong A^{\ell_{2,2}}(-2)\oplus A^{\ell_{2,3}}(-3)\oplus A^{\ell_{2,4}}(-4)\\
 &\oplus \left(A^{\ell_{1,1}}(-3)\oplus A^{\ell_{1,2}}(-4)\oplus A^{\ell_{1,3}}(-5)\right)^{\ell_{1,2}} \oplus
\left(A^{\ell_{1,1}}(-4)\oplus A^{\ell_{1,2}}(-5)\oplus A^{\ell_{1,3}}(-6)\right)^{\ell_{1,3}}\\
\BF_3 &\cong A^{\ell_{3,3}}(-3)\oplus A^{\ell_{3,4}}(-4)\oplus A^{\ell_{3,5}}(-5)\\
      &\oplus\left(A^{\ell_{1,1}}(-4)\oplus A^{\ell_{1,2}}(-5)\oplus A^{\ell_{1,3}}(-6)\right)^{\ell_{2,3}}
        \oplus\left(A^{\ell_{1,1}}(-5)\oplus A^{\ell_{1,2}}(-6)\oplus A^{\ell_{1,3}}(-7)\right)^{\ell_{2,4}}\\
      &\oplus  \left(A^{\ell_{2,2}}(-4)\oplus A^{\ell_{2,3}}(-5)\oplus A^{\ell_{2,4}}(-6)\right)^{\ell_{1,2}}\\
      &\oplus \left(A^{\ell_{1,1}}(-5)\oplus A^{\ell_{1,2}}(-6)\oplus A^{\ell_{1,3}}(-7)\right)^{(\ell_{1,2})^2}
       \oplus\left(A^{\ell_{1,1}}(-6)\oplus A^{\ell_{1,2}}(-7)\oplus A^{\ell_{1,3}}(-8)\right)^{\ell_{1,3}\ell_{1,2}}\\
      &\oplus \left(A^{\ell_{2,2}}(-5)\oplus A^{\ell_{2,3}}(-6)\oplus A^{\ell_{2,4}}(-7)\right)^{\ell_{1,3}}\\
      &\oplus \left(A^{\ell_{1,1}}(-6)\oplus A^{\ell_{1,2}}(-7)\oplus A^{\ell_{1,3}}(-8)\right)^{\ell_{1,2}\ell_{1,3}}
      \oplus \left(A^{\ell_{1,1}}(-7)\oplus A^{\ell_{1,2}}(-8)\oplus A^{\ell_{1,3}}(-9)\right)^{(\ell_{1,3})^2}. 
\end{align*}
When we view $\BF_k=\bigoplus_{s=k}^{3k} A^{u_{k,s}}(-s)$ by grouping the internal grading shifts, we then obtain the following expressions for $u_{k,s}$ in terms of $\ell_{k,r}$, which for practicality we write in terms of $\ell_i, \ell_i', \ell_i''$ here; see \eqref{ellks}. Note that $u_{k,s}=0$ for all $s>3k$ and $s<k$.  
\begin{align*}
u_{1,1}&=\ell_1,   &u_{2,2}&=\ell_2 ,              &u_{3,3}&=\ell_3, \\   
u_{1,2}&=\ell_1',  &u_{2,3}&=\ell_2'+\ell_1\ell_1',& u_{3,4}&=\ell_3'+\ell_1\ell_2'+\ell_2\ell_1'\\        
u_{1,3}&=\ell_1'', &u_{2,4}&=\ell_2''+ (\ell_1')^2+\ell_1\ell_1'',&u_{3,5}&=\ell_3''+\ell_1(\ell_1')^2+\ell_1\ell_2''+2\ell_1'\ell_2'+\ell_2\ell_1'',\\
              &&u_{2,5}&=2\ell_1'\ell_1'',    &u_{3,6}&=2\ell_1\ell_1'\ell_1'' + 2\ell_1'\ell_2'' + 2\ell_1''\ell_2' + (\ell_1')^3,\\ 
              &&  u_{2,6}&=(\ell_1'')^2, & u_{3,7}&=3(\ell_1')^2\ell_1'' + 2\ell_1''\ell_2'' + \ell_1(\ell_1'')^2,\\
              &&                         && u_{3,8}&=3\ell_1'(\ell_1'')^2\\
              &&                         && u_{3,9}&=(\ell_1'')^3.
\end{align*}
One could further express these values in terms of the invariants $a_1,a_2,a_3$ by using \Cref{ellex}.
\end{remark}

We recover the Poincar\'e series for class $\clT$, which was given in \cite[Theorem 2.1]{Av2}.

\begin{corollary} 
\label{cor: Poi}
Let $(R,\fm, \kk)$ be a local ring of embedding dimension $n$, codepth 3 and of class $\clT$, $A$ its Koszul homology, and retain the above notation. Then,
\begin{align*}
\Poi{A}{\kk}(t,z)&=\frac{1}{1-a_1tz-(a_2-3)tz^2+3t^2z^2-a_3tz^3-t^2z^3-t^3z^3}\quad\text{and}\\
\Poi{R}{\kk}(t)&=\frac{(1+t)^n}{1-a_1t^2-(a_2-3)t^3-(a_3-3)t^4-t^5-t^6}.
\end{align*}
\end{corollary}

\begin{proof}
Let $\CL(t,z) \coloneqq \sum_{k,s\ges 0} \ell_{k,s}t^kz^s$, where $\ell_{k,s}$ is defined in \eqref{ellks} for $k,s\ges 0$. By \Cref{rmk T} and \eqref{ellks}, the generating functions of the sequences
$\{\ell_{k,k}\}_{k\ges 0},\ \{\ell_{k,k+1}\}_{k\ges 0},\ \text{and}\ \{\ell_{k,k+2}\}_{k\ges 0}$ are respectively   
\[f(t)\coloneqq\frac{1}{(1-t)^3-t(a_1-3)},\quad g(t)\coloneqq\frac{t^2+t(a_2-3)}{(1-t)^3-t(a_1-3)},\text{ and} \quad h(t)\coloneqq\frac{ta_3}{(1-t)^3-t(a_1-3)}.\] 
Therefore, we obtain
\begin{equation}
\tag{$\star$}
\CL(t,z) = f(tz)+zg(tz)+z^2h(tz).
\end{equation}
By \Cref{gr Poi}, \Cref{res k over A class T}, and definition of $u_{k,s}$ in \eqref{uks}, we have
\begin{align*}
   P^A_\kk(t,z)&=\sum_{k,s\ges 0} \Bigg(\sum_{\substack{k_1+\dots+k_j=k\\ s_1+\dots+s_j=s\\  s_i\neq k_i, \forall i\ges 2}}\ell_{k_1,s_1}\dots \ell_{k_j,s_j} \Bigg) t^kz^s\\
    &= \sum_{k,s\ges 0} \ell_{k,s}t^kz^s + \sum_{k,s\ges 0}
    \sum_{\substack{ j\ges 2\\
                    k_1+\dots+k_j=k\\ 
                    s_1+\dots+s_j=s\\  
                    s_i\neq k_i, \forall i\ges 2}
                    }\left(\ell_{k_1,s_1}t^{k_1}z^{s_1}\right)\dots \left(\ell_{k_j,s_j}t^{k_j}z^{s_j}\right) \\
   &=\CL(t,z) + \sum_{j\ges 2} \big(\CL(t,z)-1\big) \big(\CL(t,z)-f(tz)\big)^{j-1}\\
   &=\CL(t,z) + \big(\CL(t,z)-1\big) \sum_{j\ges 1} \big(\CL(t,z)-f(tz)\big)^{j}\\
    &=\CL(t,z) + \big(\CL(t,z)-1\big) \left(\frac{1}{1-\CL(t,z)+f(tz)}-1\right)\\
    &=\frac{\CL(t,z)\big(1-\CL(t,z)+f(tz)\big) + \big(\CL(t,z)-1\big) \big(\CL(t,z)-f(tz)\big)}
          {1-\CL(t,z)+f(tz)}\\
    &=\frac{f(tz)}{1-\CL(t,z)+f(tz)}\\
    &\overset{\text{by } (\star)}{=}\frac{f(tz)}{1-zg(tz)-z^2h(tz)}\\
    &=\frac{1}{(1-tz)^3-tz(a_1-3)-z\big((tz)^2+tz(a_2-3)\big)-z^2(tz)a_3}\\
    &=\frac{1}{1-a_1tz-(a_2-3)tz^2+3t^2z^2-a_3tz^3-t^2z^3-t^3z^3}.
\end{align*}
In the second equality above, we separate the terms that have exactly one factor ($j=1$) and those that have more than one factor ($j \ges 2$). The third equality follows from the definition of $\CL(t,z)$, the first factor in the sum comes from $\displaystyle \CL(t,z)= 1+\sum_{k_1\ges 1}\ell_{k_1,s_1}t^{k_1}z^{s_1}$, as  $\ell_{0,0}t^0z^0 = 1$, and for $i\ges 2$ the remaining factors in the sum come from $\displaystyle \CL(t,z) = f(tz)+ \sum_{s_i\not=k_i\ges 1}\ell_{k_i,s_i}t^{k_i}z^{s_i}$, as $ f(tz)= \displaystyle \sum_{k_i\ges 1}\ell_{k_i,k_i}t^{k_i}z^{k_i}$. The remaining equalities are straightforward. 

The expression for $\Poi{R}{\kk}(t)$ follows from that of $\Poi{A}{\kk}(t,z)$ and \Cref{Poi cor}.
\end{proof}

%%%%%%%%%%%%%%%%%%%%%%%%%%%%%%%
\subsection{A minimal free resolution of $\kk$ over $R$}
\label{res over R}

In \cite{NV}, for any local ring $R$ the authors describe a truncated resolution of $\kk$ as an $R$-module up to homological degree 5. Using the descriptions of $\sC^{(k)}$ and $\varphi^{(k)}$ from \Cref{Candphi} and \Cref{cor: res over R}, we now describe the entire minimal free resolution $(F,\partial^F)$ of $\kk$ over a local ring $R$ of codepth 3 and of class $\clT$. 

\begin{theorem}
\label{res k over R class T}
  Let $(R,\fm, \kk)$ be a local ring of codepth 3 and of class $\clT$, $(K,\partial^K)$ the Koszul complex of $R$ on a minimal set of generators of $\fm$, and retain the above notation. Then, a minimal free resolution $F$ of the residue field $\kk$ over $R$ is given as follows. 
  \begin{enumerate}[\((a)\)]
      \item For each $k \ges 0$, the $k$-th homological component is \[F_k \cong \bigoplus_{i=0}^k \Bigg(\bigoplus_{\substack{\deg_1\bm+\deg_2\bm}=k-i}K^{\deg_3\bm}_i\Bigg),\]
where $\bm$ is a monomial of $\TF$ as in \eqref{algebraF}, with different degrees as defined in \eqref{deg}. 
\item For each $k \ges 0$, the differential maps $\partial_k^F$ consists of maps described below. If we write  $\bm= X_{j-1,j-1}\bn$ for some $j\ges 1$, where $\bn$ is a monomial of $\TF$ that does not contain $X_{u,u}$ for any $u \ges 1$, and consider $K_i^{\deg_3\bm}$ a component of $F_{k-1}$, then all the maps going into $K_i^{\deg_3\bm}$ have the following form:

\begin{equation*}  
 \xymatrixrowsep{3pc}
 \xymatrixcolsep{5pc}
 \xymatrix{
  K_{i+1}^{\deg_3\bm}\ar@{->}[r]^{(\partial_{i+1}^K)^{\deg_3\bm} }\ar@{}[d]_{\bigoplus}
 &K_i^{\deg_3\bm}, \\ 
(K_{i-1}^{\ell_{j}})   ^{\deg_3\bn}\ar@{->}[ru]  |{(\alpha_{j})   ^{\deg_3\bn}}\ar@{}[d]_{\bigoplus}&\\
(K_{i-2}^{\ell_{j}'}) ^{\deg_3\bn}\ar@{->}[ruu] |{(\alpha_{j}') ^{\deg_3\bn}}\ar@{}[d]_{\bigoplus}&\\
(K_{i-3}^{\ell_{j}''}) ^{\deg_3\bn}\ar@{->}[ruuu]|{(\alpha_{j}'') ^{\deg_3\bn}}&
}
\end{equation*}
where the sequences of integers $\ell_j, \ell_j', \ell_j''$ 
and the matrices $\alpha_j, \alpha_j', \alpha_j''$ with entries in $K_1,K_2,K_3$,  respectively, are given in \Cref{def seq T}.
 \end{enumerate}
\end{theorem}
\begin{proof}
Remark that $\deg_3\bm=\ell_{j-1}\deg_3\bn$.
For all $k \ges 0$, we have
  \begin{align*}
      F_k = \bigoplus_{j=0}^{\lfloor\frac{k}{2}\rfloor}\sC^{(j)}_{k-j}&= \bigoplus_{j\ges 0}\sC^{(j)}_{k-j} = \bigoplus_{j\ges 0}\Big(\bigoplus_{\deg_1\bm=j}  \Sigma^{\deg_2\bm} K^{\deg_3\bm}\Big)_{k-j}= \bigoplus_{j\ges 0}\Big(\bigoplus_{\deg_1\bm=j}   K^{\deg_3\bm}_{k-j-\deg_2\bm}\Big),
  \end{align*} 
where $\bm$ is a monomial of $\TF$. The first equality follows from \Cref{cor: res over R}, the second equality follows from the fact that $\sC^{(j)}_{k-j}=0$ for all $k-j<j$, and the third equality follows from \eqref{def Ck}, and we apply the homological shift in the last equality.
Set $i=k-j-\deg_2\bm$, we obtain the desired expression for $F_k$.

The differential maps $\partial^F_k$ follows from \Cref{cor: res over R} and descriptions of $\sC^{(k)}$ and $\varphi^{(k)}$ in \Cref{Candphi}.
\end{proof}

\begin{remark}
\label{rmk:res}
We illustrate the resolution $F$ in terms of Koszul blocks, up to degree 7. By \Cref{cor: res over R},
\begin{align*}
F_0&=\sC^{(0)}_0&F_2&=\sC^{(0)}_2\oplus \sC^{(1)}_1&F_4&=\sC^{(0)}_4 \oplus \sC^{(1)}_3 \oplus \sC^{(2)}_2&F_6&=\sC^{(0)}_6 \oplus \sC^{(1)}_5 \oplus \sC^{(2)}_4 \oplus \sC^{(3)}_3\\
F_1&=\sC^{(0)}_1&F_3&=\sC^{(0)}_3\oplus \sC^{(1)}_2&F_5&=\sC^{(0)}_5 \oplus \sC^{(1)}_4 \oplus \sC^{(2)}_3&F_7&=\sC^{(0)}_7 \oplus \sC^{(1)}_6 \oplus \sC^{(2)}_5 \oplus \sC^{(3)}_4.
\end{align*}
Using \eqref{def Ck} and the labels in \eqref{ellks}, we describe the components of $F$ in terms of $K$ and $\ell_{k,s}$:
\begin{align*}
F_0&=K_0^{\ell_{0,0}}\\
F_1&=K_1^{\ell_{0,0}} \\
F_2&=K_2^{\ell_{0,0}}\oplus K_0^{\ell_{1,1}}\\
F_3&=K_3^{\ell_{0,0}}\oplus (K_1^{\ell_{1,1}}\oplus   K_0^{\ell_{1,2}})\\
F_4&=K_4^{\ell_{0,0}}\oplus (K_2^{\ell_{1,1}}\oplus K_1^{\ell_{1,2}} \oplus K_0^{\ell_{1,3}}) \oplus K_0^{\ell_{2,2}}\\
F_5&=K_5^{\ell_{0,0}}\oplus (K_3^{\ell_{1,1}}\oplus K_2^{\ell_{1,2}} \oplus K_1^{\ell_{1,3}}) \oplus (K_1^{\ell_{2,2}}\oplus K_0^{\ell_{2,3}} \oplus K_0^{\ell_{1,1}\ell_{1,2}}) \\
F_6&=K_6^{\ell_{0,0}}\oplus (K_4^{\ell_{1,1}}\oplus K_3^{\ell_{1,2}} \oplus K_2^{\ell_{1,3}}) \oplus (K_2^{\ell_{2,2}}\oplus K_1^{\ell_{2,3}} \oplus K_0^{\ell_{2,4}} \oplus K_1^{\ell_{1,1}\ell_{1,2}}\oplus K_0^{(\ell_{1,2})^2}\oplus K_0^{\ell_{1,1}\ell_{1,3}}) \oplus K_0^{\ell_{3,3}}\\
F_7&=K_7^{\ell_{0,0}}\oplus (K_5^{\ell_{1,1}}\oplus K_4^{\ell_{1,2}} \oplus K_3^{\ell_{1,3}}) \\
   &\quad \oplus (K_3^{\ell_{2,2}}\oplus K_2^{\ell_{2,3}} \oplus K_1^{\ell_{2,4}} \oplus K_2^{\ell_{1,1}\ell_{1,2}}\oplus K_1^{(\ell_{1,2})^2} \oplus K_0^{\ell_{1,3}\ell_{1,2}}
   \oplus K_1^{\ell_{1,1}\ell_{1,3}} \oplus K_0^{\ell_{1,2}\ell_{1,3}}) \\
   &\quad \oplus (K_1^{\ell_{3,3}} \oplus K_0^{\ell_{3,4}} \oplus K_0^{\ell_{1,1}\ell_{2,3}} \oplus K_0^{\ell_{2,2}\ell_{1,2}}).
\end{align*}

Using the labels in \eqref{ellks} and \eqref{alphaks}, a portion of the differential map $\partial^F_k$ given in \Cref{res k over R class T} is
\begin{equation*}  
 \xymatrixrowsep{3pc}
 \xymatrixcolsep{5pc}
 \xymatrix{
  K_{i+1}^{\deg_3\bm}\ar@{->}[r]^{(\partial_{i+1}^K)^{\deg_3\bm} }\ar@{}[d]_{\bigoplus}
 &K_i^{\deg_3\bm}. \\ 
(K_{i-1}^{\ell_{j,j}})   ^{\deg_3\bn}\ar@{->}[ru]  |{(\alpha_{j,j})   ^{\deg_3\bn}}\ar@{}[d]_{\bigoplus}&\\
(K_{i-2}^{\ell_{j,j+1}}) ^{\deg_3\bn}\ar@{->}[ruu] |{(\alpha_{j,j+1}) ^{\deg_3\bn}}\ar@{}[d]_{\bigoplus}&\\
(K_{i-3}^{\ell_{j,j+2}}) ^{\deg_3\bn}\ar@{->}[ruuu]|{(\alpha_{j,j+2}) ^{\deg_3\bn}}&
}
\end{equation*}

Set $\partial\coloneqq\partial^K$ for the differential maps of the Koszul complex. The differential maps of $F$, using the above ordering, are given by blocks of maps: 
\begin{align*}
    \partial_1^F& = \partial_1 \\[0.05in]
    \partial_2^F& =\begin{pmatrix} \partial_2&\alpha_{1,1} 
                   \end{pmatrix} \\[0.05in]
    \partial_3^F& = 
    \begin{pmatrix}
    \partial_3&\alpha_{1,1}&\alpha_{1,2}\\
            0&\partial_1&0
    \end{pmatrix} \\[0.05in]
     \partial_4^F& = 
    \begin{pmatrix}
    \partial_4&\alpha_{1,1}&\alpha_{1,2}&\alpha_{1,3}&0\\
            0&\partial_2^{\ell_{1,1}}&0&0&\alpha_{2,2}\\
            0&0&\partial_1^{\ell_{1,2}}&0&0
    \end{pmatrix} \\[0.05in]
     \partial_5^F& =  \begin{pmatrix}
     \partial_5&\alpha_{1,1}&\alpha_{1,2}&\alpha_{1,3}&0&0&0\\
            0&\partial_3^{\ell_{1,1}}&0&0&\alpha_{2,2}&\alpha_{2,3}&0\\
            0&0&\partial_2^{\ell_{1,2}}&0&0&0&\alpha_{1,1}^{\ell_{1,2}}\\
            0&0&0&\partial_1^{\ell_{1,3}}&0&0&0\\
             0&0&0&0&\partial_1^{\ell_{2,2}}&0&0\\
    \end{pmatrix} \\[0.05in]
\partial_6^F&=  
\left(
\arraycolsep=2pt\def\arraystretch{1.5}
\begin{array}{ccccccccccc}
\partial_6&\alpha_{1,1}&\alpha_{1,2}&\alpha_{1,3}&0&0&0&0&0&0&0\\
0&\partial_4^{\ell_{1,1}}&0&0&\alpha_{2,2}&\alpha_{2,3}&\alpha_{2,4}&0&0&0&0\\  
0&0&\partial_3^{\ell_{1,2}}&0&0&0&0&\alpha_{1,1}^{\ell_{1,2}}&\alpha_{1,2}^{\ell_{1,2}}&0&0 \\0&0&0&\partial_3^{\ell_{1,3}}&0&0&0&0&0&\alpha_{1,1}^{\ell_{1,3}}&0\\
0&0&0&0&\partial_2^{\ell_{2,2}}&0&0&0&0&0&\alpha_{3,3}\\
0&0&0&0&0&\partial_1^{\ell_{2,3}}&0&0&0&0&0 \\
0&0&0&0&0&0&0&\partial_1^{\ell_{1,1}\ell_{1,2}}&0&0&0
\end{array} \right) \\[0.05in] 
\partial_7^F &=\left(
\arraycolsep=2pt\def\arraystretch{1.5}
\begin{array}{cccccccccccccccc}
\partial_7&\alpha_{1,1}&\alpha_{1,2}&\alpha_{1,3}&0&0&0&0&0&0&0&0&0&0&0&0\\
0&\partial_5^{\ell_{1,1}}&0&0&\alpha_{2,2}&\alpha_{2,3}&\alpha_{2,4}&0&0&0&0&0&0&0&0&0\\
0&0&\partial_4^{\ell_{1,2}}&0&0&0&0&\alpha_{1,1}^{\ell_{1,2}}&\alpha_{1,2}^{\ell_{1,2}}&\alpha_{1,3}^{\ell_{1,3}}&0&0&0&0&0&0\\
0&0&0&\partial_3^{\ell_{1,3}}&0&0&0&0&0&0&\alpha_{1,1}^{\ell_{1,3}}&\alpha_{1,2}^{\ell_{1,3}}&0&0&0&0 \\
0&0&0&0&\partial_3^{\ell_{2,2}}&0&0&0&0&0&0&0&\alpha_{3,3}&\alpha_{3,4}&0&0\\
0&0&0&0&0&\partial_2^{\ell_{2,3}}&0&0&0&0&0&0&0&0&\alpha_{1,2}^{\ell_{23}}&0\\
0&0&0&0&0&0&\partial_1^{\ell_{2,4}}&0&0&0&0&0&0&0&0&0\\
0&0&0&0&0&0&0&\partial_2^{\ell_{1,1}\ell_{1,2}}&0&0&0&0&0&0&0&\alpha_{2,2}^{\ell_{1,2}}\\
0&0&0&0&0&0&0&0&\partial_1^{\ell_{1,2}^2}&0&0&0&0&0&0&0\\
0&0&0&0&0&0&0&0&0&0&\partial_1^{\ell_{1,1}\ell_{1,3}}&0&0&0&0&0\\
0&0&0&0&0&0&0&0&0&0&0&0&\partial_1^{\ell_{3,3}}&0&0&0\\
 \end{array}\right).
\end{align*}
\end{remark}

%%%%%%%%%%%%%%%%%%%%%%%%%%%%%%%
%%%%%%%%%%%%%%%%%%%%%%%%%%%%%%%
\section{A class $\clT$ example}
\label{sec:example}

In this section we discuss in detail an example of class $\clT$ that is a codepth 3 almost complete intersection ring. In fact, the almost complete intersections of codepth 3 of odd type are all of class $\clT$, see \cite[Theorem 2.3]{CVW}. Here we illustrate our results from \Cref{class T} with an example that has the smallest possible number of generators of the defining ideal (that is 4), with smallest possible type (that is 3). 

Let $\kk$ be a field, $Q=\kk[x,y,z]$ and $R \coloneqq Q/ (x^2,y^2,z^2,xyz)$. A minimal free resolution of $R$ as a $Q$-module has the form $0\to Q^3\to Q^6\to Q^4\to Q$, so we have 
\[a_0=1,\quad a_1=4,\quad a_2=6,\quad\text{and}\quad a_3=3.\]
By \Cref{cor: Poi} we have 
$\displaystyle \Poi{R}{\kk}(t)=\frac{(1+t)^3}{1-4t^2-3t^3-t^5-t^6}$. Therefore, a minimal free resolution of the $R$-module  $\kk$ has the form: 
\begin{equation}
\label{res ex}
F:\quad  \cdots \to R^{465} \xra{\partial^F_7} R^{200} \xra{\partial^F_6} R^{86} \xra{\partial^F_5} R^{37} \xra{\partial^F_4} R^{16} \xra{\partial^F_3} R^7 \xra{\partial^F_2} R^3 \xra{\partial^F_1} R^1. 
\end{equation}

We now describe how our results yield this resolution $F$ explicitly via Koszul blocks, and concretely express the differential maps $\partial^F_k$ for $1\les k \les 7$. By \Cref{rmk T} and \eqref{ellks}, the sequences of integers $\{b_k\}_{k\ges 0}$, $\{\ell_{k,r}\}_{k\ges 0}$ for $r\in\{k,k+1,k+2\}$ for this example are given by:
\begin{align*}
\{b_k\}_{k\ges 0}& =\{1,3,6,10,15,21,\dots\}\\ 
\{\ell_{k,k}\}_{k\ges 0}& =\{1,4,13,41,129,406,\dots\}\\ 
\{\ell_{k,k+1}\}_{k\ges 0}& =\{0,3,13,43,136,1347,\dots\}\\ 
\{\ell_{k,k+2}\}_{k\ges 0}& =\{0,3,12,39,123,387,1218\dots\}. 
\end{align*}

The Koszul complex $(K,\partial)$ is given by
\(K: 0 \to K_3 \xrightarrow{\partial_3} K_2 \xrightarrow{\partial_2} K_1 \xrightarrow{\partial_1} K_0 \to 0\). Thus, by \Cref{rmk:res} the components of the resolution $F$ up to degree 7, are given by:
\begin{align*}
F_0&\coloneqq K_0=R\\
F_1&\coloneqq K_1=R^3\\
F_2&\coloneqq K_2\oplus K_0^4=R^3\oplus R^4\cong R^7\\
F_3&\coloneqq K_3\oplus K_1^4 \oplus K_0^3=R\oplus (R^3)^4\oplus R^3\cong R^{16}\\
F_4&\coloneqq K_2^4\oplus K_1^{3}\oplus K_0^3\oplus K_0^{13}=(R^3)^4\oplus(R^3)^3\oplus R^3\oplus R^{13}\cong R^{37}\\
F_5&\coloneqq K_3^4\oplus K_2^{3}\oplus K_1^{3}\oplus K_1^{13}\oplus K_0^{13}\oplus K_0^{12}=R^4\oplus (R^3)^{3}\oplus (R^3)^{3}\oplus (R^3)^{13}\oplus R^{13}\oplus R^{12}\cong R^{86}\\
F_6&\coloneqq K_3^{3}\oplus K_2^{3}\oplus K_2^{13}\oplus K_1^{13}\oplus K_1^{12}\oplus K_0^{12}\oplus K_0^9\oplus K_0^{12}\oplus K_0^{41}\\
&=R^{3}\oplus (R^3)^{3}\oplus (R^3)^{13}\oplus (R^3)^{13}\oplus (R^3)^{12}
\oplus R^{12}\oplus R^9\oplus R^{12}\oplus R^{41}\cong R^{200}\\
F_7&= K_3^3 \oplus K_3^{13}\oplus K_2^{13}\oplus K_1^{12}\oplus K_2^{4\cdot 3}\oplus K_1^{3^2}\oplus K_0^{3\cdot 3}\oplus K_1^{4\cdot 3} \oplus K_0^{3\cdot 3}\oplus K_1^{41} \oplus K_0^{43} \oplus K_0^{4\cdot 13} \oplus K_0^{13\cdot 3}\\
&=R^3\oplus R^{13}\oplus (R^3)^{13}\oplus (R^3)^{12}\oplus (R^3)^{12} \oplus (R^3)^{9}\oplus R^9\oplus (R^3)^{12}\oplus R^9\oplus (R^3)^{41}\oplus R^{43}\oplus R^{52}\oplus R^{39} \cong R^{465}.
\end{align*}

By \Cref{rmk:res}, the differentials of the free resolution $(F,\partial^F)$ up to degree 7  involve the following maps: 
\begin{align*}
\partial_1,\quad \partial_2,\quad \partial_3,\quad \alpha_{1,1}, \quad \alpha_{1,2}, \quad \alpha_{1,3},\quad
\alpha_{2,2}, \quad \alpha_{2,3}, \quad \alpha_{2,4},\quad \alpha_{3,3}, \quad \text{and}\quad \alpha_{3,4}.
\end{align*}

Let $\{e_1,e_2,e_3\}$ be a basis for $K_1=R^3$, then $\{e_{12}, e_{13}, e_{23}\}$ is a basis for $K_2=R^3$, where $e_{ij} = e_1 \wedge e_j$, and $\{e_{123}\}$ is a basis for $K_3=R$, where $e_{123}=e_1 \wedge e_2 \wedge e_3$. The differential maps in the Koszul complex are:
\begin{align*}
    \partial_1(e_1)&=x   &\partial_2(e_{12})&=-ye_1+xe_2 & \partial_3(e_{123})&=ze_{12}-ye_{13}+xe_{23}. \\
    \partial_1(e_2)&=y & \partial_2(e_{13})&=-ze_1+xe_3&&  \\
    \partial_1(e_3)&=z & \partial_2(e_{23})&=-ze_2+ye_3 &&
\end{align*}
Set $A\coloneqq \HH{}{K}=\kk\oplus A_1\oplus A_2\oplus A_3$. Using \texttt{Macaulay2} \cite{M2}, the components $A_i$ have the following $\kk$-bases:
\begin{align*}
&A_0 \colon& [1], &&&&&&& \\[0.05in]
&A_1 \colon& [z^1_1]&=[xe_1], &[z^1_2]&=[ye_2], &[z^1_3]&=[ze_3], & [z^1_4]&=[yze_1], \\[0.05in]
&A_2 \colon& [z^1_1]\cdot[z^1_2]&=[xye_{12}], &  [z^1_2]\cdot[z^1_3]&=[yze_{23}], & [z^1_1]\cdot[z^1_3]&=[xze_{13}], & \text{and} &\\
&& [z^2_1]&=[yze_{12}], & [z^2_2]&=[xze_{12}], &
  [z^2_3]&=[yze_{13}], && \\[0.05in] 
&A_3 \colon& [z^3_1]&=[yze_{123}], & [z^3_2]&=[xze_{123}], & [z^3_3]&=[xye_{123}]. &&
\end{align*}

\Cref{rmk: ex beta} gives  the expressions for $\beta_i,\beta_i'$ for $0 \les i \les 3$ directly and by \eqref{def gamma} we have:
\begin{equation*}
\gamma_1=\begin{pmatrix}z^1_4\end{pmatrix}, \qquad \gamma_2=\begin{pmatrix}z^2_1&z^2_2&z^2_3\end{pmatrix}, \qquad\text{ and }\qquad
\gamma_3=\begin{pmatrix}z^3_1&z^3_2&z^3_3\end{pmatrix}.
\end{equation*}
By \Cref{def seq T}, and \eqref{alphaks}, we obtain the maps $\alpha_{k,r}$  for $1\les k\les 3$ and $r\in\{k,k+1,k+2\}$ as follows:
\begin{align*}
\alpha_{1,1}&=\left(\begin{array}{ccc|c}z^1_1&z^1_2&z^1_3&z^1_4\end{array}\right)\\[0.05in]
    \alpha_{1,2}&=\begin{pmatrix}z^2_1&z^2_2&z^2_3\end{pmatrix}\\[0.05in]
    \alpha_{1,3}&=\begin{pmatrix}z^3_1&z^3_2&z^3_3\end{pmatrix}\\[0.05in]
    \alpha_{2,2}&=
    \left(\begin{array}{cccccc:ccc|cccc}
    z^1_1 & z^1_2 & z^1_3 & 0 & 0 & 0 & 0 & 0 & 0 &z^1_4 & 0   & 0 & 0\\
    0 & z^1_1 & 0 & z^1_2 & z^1_3 & 0 & 0 & 0 & 0 &  0   &z^1_4& 0 & 0\\
    0 & 0 & z^1_1 & 0 & z^1_2 & z^1_3 & 0 & 0 & 0 & 0    & 0   &z^1_4& 0\\
    0 & 0 & 0     & 0 & 0     &   0   & z^1_1 & z^1_2 & z^1_3 & 0 & 0 & 0 &z^1_4
\end{array}\right)\\[0.05in]
\alpha_{2,3}&= 
\left(\begin{array}{c|ccc:ccc:ccc}
    z^1_2 \wedge z^1_3&z^2_1&z^2_2&z^2_3&0&0&0&0&0&0 \\ 
    z^1_1 \wedge z^1_3&0&0&0&z^2_1&z^2_2&z^2_3&0&0&0\\ 
    z^1_1 \wedge z^1_2&0&0&0&0&0&0&z^2_1&z^2_2&z^2_3
\end{array}\right)
\\[0.05in]
\alpha_{2,4}&=\left(\begin{array}{ccc:ccc:ccc:ccc}
z^3_1&z^3_2&z^3_3&0    &0    &0    &0    &0    &0    &0    &0    &0\\
0    &0    &0    &z^3_1&z^3_2&z^3_3&0    &0    &0    &0    &0    &0\\
0    &0    &0    &0    &0    &0    &z^3_1&z^3_2&z^3_3&0    &0    &0\\
0    &0    &0    &0    &0    &0    &0    &0    &0    &z^3_1&z^3_2&z^3_3
\end{array}\right)
\end{align*}

\begin{landscape} 
\begin{align*}
\alpha_{3,3}= & \fontsize{7.5}{7.5}\selectfont{
\left(
\arraycolsep=1.5pt\def\arraystretch{2}
\begin{array}{cccccccccc:cccccc:ccc:ccc:ccc:ccc|ccccccccccccc}
    z^1_1 & z^1_2 & z^1_3 & 0 & 0 & 0 & 0 & 0 & 0 & 0 & 0 & 0 & 0 & 0 & 0 & 0 & 0 & 0 & 0 & 0 & 0 & 0 & 0 & 0 & 0  &   0 & 0 & 0     &z^1_4 & 0 &  0 & 0 & 0 & 0 & 0 & 0 & 0 & 0 & 0 & 0 & 0\\
    0 & z^1_1 & 0 & z^1_2 & z^1_3 & 0 & 0 & 0 & 0 & 0 & 0 & 0 & 0 & 0 & 0 & 0 & 0 & 0 & 0 & 0 & 0 & 0 & 0 & 0 & 0 &    0 & 0 & 0     & 0  &z^1_4 &  0 & 0 & 0 & 0 & 0 & 0 & 0 & 0 & 0 & 0 & 0\\
    0& 0& z^1_1 & 0 & z^1_2 & z^1_3 & 0 & 0 & 0 & 0 & 0 & 0 & 0 & 0 & 0 & 0 & 0 & 0 & 0 & 0 & 0 & 0 & 0 & 0 & 0  &     0 & 0 & 0     & 0 &  0 &z^1_4  & 0 & 0 & 0 & 0 & 0 & 0 & 0 & 0 & 0 & 0\\
    0 & 0 & 0 & z^1_1 & 0 & 0 & z^1_2 & z^1_3 & 0 & 0 & 0 & 0 & 0 & 0 & 0 & 0 & 0 & 0 & 0 & 0 & 0 & 0 & 0 & 0 & 0 &    0 & 0 & 0     & 0 &  0 & 0  &z^1_4 & 0 & 0 & 0 & 0 & 0 & 0 & 0 & 0 & 0\\
    0 & 0 & 0 & 0 & z^1_1 & 0 & 0 & z^1_2 & z^1_3 & 0 & 0 & 0 & 0 & 0 & 0 & 0 & 0 & 0 & 0 & 0 & 0 & 0 & 0 & 0 & 0  &   0 & 0 & 0     & 0 &  0 & 0 & 0  &z^1_4  & 0 & 0 & 0 & 0 & 0 & 0 & 0 & 0\\
    0 & 0 & 0 & 0 & 0 & z^1_1 & 0 & 0 & z^1_2 & z^1_3 & 0 & 0 & 0 & 0 & 0 & 0 & 0 & 0 & 0 & 0 & 0 & 0 & 0 & 0 & 0 &    0 & 0 & 0     & 0 &  0 & 0 & 0 & 0 &z^1_4 & 0 & 0 & 0 & 0 & 0 & 0 & 0\\
    \hdashline
    0 & 0 & 0 & 0& 0& 0& 0& 0 &0& 0&z^1_1 & z^1_2 & z^1_3 & 0 & 0 & 0 & 0 & 0 & 0 & 0 & 0 & 0 & 0 & 0 & 0&             0 & 0 & 0   & 0 &  0 & 0 & 0 & 0 & 0&z^1_4 & 0 & 0 & 0 & 0 & 0 & 0 \\
    0 & 0 & 0 & 0& 0& 0& 0& 0 &0& 0&0 & z^1_1 & 0 & z^1_2 & z^1_3 & 0 & 0 & 0 & 0 & 0 & 0 & 0 & 0 & 0 & 0 & 0 & 0 & 0              & 0 &  0 & 0 & 0 & 0 & 0 & 0 &z^1_4  & 0 & 0 & 0 & 0 & 0\\
    0 & 0 & 0 & 0& 0& 0& 0& 0 &0& 0&0 & 0 & z^1_1 & 0 & z^1_2 & z^1_3 & 0 & 0 & 0 & 0 & 0 & 0 & 0 & 0 & 0 & 0 & 0 & 0              & 0 &  0 & 0 & 0 & 0 & 0 & 0 & 0  &z^1_4& 0 & 0 & 0 & 0\\
   \hdashline
    0 & 0 & 0 & 0& 0& 0& 0& 0 &0&0   &0 & 0 & 0 & 0 & 0 & 0 &   z^1_1 &z^1_2 & z^1_3 &0 & 0 & 0 & 0 & 0 & 0 & 0 & 0 & 0               & 0 &  0 & 0 & 0 & 0 & 0 & 0 & 0 & 0  &z^1_4 & 0 & 0 & 0\\
    0 & 0 & 0 & 0& 0& 0& 0& 0 &0&0   &0 & 0 & 0 & 0 & 0 & 0 &   0 & 0 & 0 & z^1_1 &z^1_2 & z^1_3 & 0 & 0 & 0 & 0 & 0 & 0               & 0 &  0 & 0 & 0 & 0 & 0 & 0 & 0 & 0 & 0  &z^1_4& 0 & 0\\
    0 & 0 & 0 & 0& 0& 0& 0& 0 &0&0   &0 & 0 & 0 & 0 & 0 & 0 &   0 & 0 & 0  & 0 & 0 & 0& z^1_1 &z^1_2 & z^1_3 & 0 & 0 & 0              & 0 &  0 & 0 & 0 & 0 & 0 & 0 & 0 & 0 & 0 & 0 &z^1_4 & 0\\
    0 & 0 & 0 & 0& 0& 0& 0& 0 &0&0   &0 & 0 & 0 & 0 & 0 & 0 &   0 & 0 & 0 & 0 & 0 & 0& 0 & 0 & 0 & z^1_1 &z^1_2 & z^1_3               & 0 &  0 & 0 & 0 & 0 & 0 & 0 & 0 & 0 & 0 & 0 & 0 &z^1_4   
\end{array} \right)
}\\[0.05in]
\alpha_{3,4} = &\fontsize{7.5}{7.5}\selectfont{
\left(
\arraycolsep=1.5pt\def\arraystretch{2.2}
\begin{array}{ccc:c|ccc:ccc:ccc:ccc:ccc:ccc:ccc:ccc:ccc:ccc:ccc:ccc:ccc}
    0 & 0 & 0 & 0 & z^2_1 & z^2_2 & z^2_3 & 0 & 0 & 0 & 0 & 0 & 0 & 0 & 0 & 0 & 0 & 0 & 0 & 0 & 0 & 0 & 0 & 0 & 0 & 0 & 0 & 0 & 0 & 0 & 0 & 0 & 0 & 0 & 0 & 0 & 0 & 0 & 0 & 0 & 0 & 0 & 0\\
    0 & 0 & 0 & 0 & 0 & 0 & 0 & z^2_1 & z^2_2 & z^2_3 & 0 & 0 & 0 & 0 & 0 & 0 & 0 & 0 & 0 & 0 & 0 & 0 & 0 & 0 & 0 & 0 & 0 & 0 & 0 & 0 & 0 & 0 & 0 & 0 & 0 & 0 & 0 & 0 & 0 & 0 & 0 & 0 & 0\\
    0 & 0 & 0 & 0 & 0 & 0 & 0 & 0 & 0 & 0 & z^2_1 & z^2_2 & z^2_3 & 0 & 0 & 0 & 0 & 0 & 0 & 0 & 0 & 0 & 0 & 0 & 0 & 0 & 0 & 0 & 0 & 0 & 0 & 0 & 0 & 0 & 0 & 0 & 0 & 0 & 0 & 0 & 0 & 0 & 0\\ 
    0 & 0 & 0 & 0 & 0 & 0 & 0 & 0 & 0 & 0 & 0 & 0 & 0 & z^2_1 & z^2_2 & z^2_3 & 0 & 0 & 0 & 0 & 0 & 0 & 0 & 0 & 0 & 0 & 0 & 0 & 0 & 0 & 0 & 0 & 0 & 0 & 0 & 0 & 0 & 0 & 0 & 0 & 0 & 0 & 0\\ \hdashline 
    z^1_2 \wedge z^1_3 & 0 & 0 & 0 & 0 & 0 & 0 & 0 & 0 & 0 & 0 & 0 & 0 & 0 & 0 & 0 & z^2_1 & z^2_2 & z^2_3 & 0 & 0 & 0 & 0 & 0 & 0 & 0 & 0 & 0 & 0 & 0 & 0 & 0 & 0 & 0 & 0 & 0 & 0 & 0 & 0 & 0 & 0 & 0 & 0\\
    z^1_1 \wedge z^1_3 & z^1_2 \wedge z^1_3 & 0 & 0 & 0 & 0 & 0 & 0 & 0 & 0 & 0 & 0 & 0 & 0 & 0 & 0 & 0 & 0 & 0 & z^2_1 & z^2_2 & z^2_3 & 0 & 0 & 0 & 0 & 0 & 0 & 0 & 0 & 0 & 0 & 0 & 0 & 0 & 0 & 0 & 0 & 0 & 0 & 0 & 0 & 0\\
    z^1_1 \wedge z^1_2 & 0 & z^1_2 \wedge z^1_3 & 0 & 0 & 0 & 0 & 0 & 0 & 0 & 0 & 0 & 0 & 0 & 0 & 0 & 0 & 0 & 0 & 0 & 0 & 0 & z^2_1 & z^2_2 & z^2_3 & 0 & 0 & 0 & 0 & 0 & 0 & 0 & 0 & 0 & 0 & 0 & 0 & 0 & 0 & 0 & 0 & 0 & 0\\
    0 & z^1_1 \wedge z^1_3 & 0 & 0 & 0 & 0 & 0 & 0 & 0 & 0 & 0 & 0 & 0 & 0 & 0 & 0 & 0 & 0 & 0 & 0 & 0 & 0 & 0 & 0 & 0 & z^2_1 & z^2_2 & z^2_3 & 0 & 0 & 0 & 0 & 0 & 0 & 0 & 0 & 0 & 0 & 0 & 0 & 0 & 0 & 0\\
    0 & z^1_1 \wedge z^1_2 & z^1_1 \wedge z^1_3 & 0 & 0 & 0 & 0 & 0 & 0 & 0 & 0 & 0 & 0 & 0 & 0 & 0 & 0 & 0 & 0 & 0 & 0 & 0 & 0 & 0 & 0 & 0 & 0 & 0 & z^2_1 & z^2_2 & z^2_3 & 0 & 0 & 0 & 0 & 0 & 0 & 0 & 0 & 0 & 0 & 0 & 0\\
    0 & 0 & z^1_1 \wedge z^1_2 & 0 & 0 & 0 & 0 & 0 & 0 & 0 & 0 & 0 & 0 & 0 & 0 & 0 & 0 & 0 & 0 & 0 & 0 & 0 & 0 & 0 & 0 & 0 & 0 & 0 & 0 & 0 & 0 & z^2_1 & z^2_2 & z^2_3 & 0 & 0 & 0 & 0 & 0 & 0 & 0 & 0 & 0\\
    0 & 0 & 0 & z^1_2 \wedge z^1_3 & 0 & 0 & 0 & 0 & 0 & 0 & 0 & 0 & 0 & 0 & 0 & 0 & 0 & 0 & 0 & 0 & 0 & 0 & 0 & 0 & 0 & 0 & 0 & 0 & 0 & 0 & 0 & 0 & 0 & 0 & z^2_1 & z^2_2 & z^2_3 & 0 & 0 & 0 & 0 & 0 & 0\\
    0 & 0 & 0 & z^1_1 \wedge z^1_3 & 0 & 0 & 0 & 0 & 0 & 0 & 0 & 0 & 0 & 0 & 0 & 0 & 0 & 0 & 0 & 0 & 0 & 0 & 0 & 0 & 0 & 0 & 0 & 0 & 0 & 0 & 0 & 0 & 0 & 0 & 0 & 0 & 0 & z^2_1 & z^2_2 & z^2_3 & 0 & 0 & 0\\
    0 & 0 & 0 & z^1_1 \wedge z^1_2 & 0 & 0 & 0 & 0 & 0 & 0 & 0 & 0 & 0 & 0 & 0 & 0 & 0 & 0 & 0 & 0 & 0 & 0 & 0 & 0 & 0 & 0 & 0 & 0 & 0 & 0 & 0 & 0 & 0 & 0 & 0 & 0 & 0 & 0 & 0 & 0 & z^2_1 & z^2_2 & z^2_3\\
\end{array}\right) 
}
\end{align*}
\end{landscape}

In terms of Koszul blocks, the differential maps  $\{\partial^F_{k}\}_{1\les k\les 7}$ of \eqref{res ex} are given by:
\begin{align*}
    \partial_1^F& = \partial_1 \\[0.05in]
    \partial_2^F& =\begin{pmatrix} \partial_2&\alpha_{1,1} 
                   \end{pmatrix} \\[0.05in]
    \partial_3^F& = 
    \begin{pmatrix}
    \partial_3&\alpha_{1,1}&\alpha_{1,2}\\
            0&\partial_1&0
    \end{pmatrix} \\[0.05in]
     \partial_4^F& = 
    \begin{pmatrix}
    \alpha_{1,1}&\alpha_{1,2}&\alpha_{1,3}&0\\
            \partial_2^4&0&0&\alpha_{2,2}\\
            0&\partial_1^3&0&0
    \end{pmatrix} \\[0.05in]
     \partial_5^F& =  \begin{pmatrix}
            \partial_3^4&0&0&\alpha_{2,2}&\alpha_{2,3}&0\\
            0&\partial_2^3&0&0&0&\alpha_{1,1}^3\\
            0&0&\partial_1^3&0&0&0\\
             0&0&0&\partial_1^{13}&0&0\\
    \end{pmatrix} \\[0.05in]
\partial_6^F&=  
\left(
\arraycolsep=2pt\def\arraystretch{1.5}
\begin{array}{ccccccccc}
0&0&\alpha_{2,2}&\alpha_{2,3}&\alpha_{2,4}&0&0&0&0\\  
\partial_3^3&0&0&0&0&\alpha_{1,1}^3&\alpha_{1,2}^3&0&0 \\
0&\partial_3^3&0&0&0&0&0&\alpha_{1,1}^{3}&0\\
0&0&\partial_2^{13}&0&0&0&0&0&\alpha_{3,3}\\
0&0&0&\partial_1^{13}&0&0&0&0&0 \\
0&0&0&0&0&\partial_1^{12}&0&0&0
\end{array} \right) \\[0.05in] 
\partial_7^F &=
\left(
\arraycolsep=1.5pt\def\arraystretch{1.5}
\begin{array}{ccccccccccccc}
0&0&0&0&\alpha_{1,1}^{3}&\alpha_{1,2}^{3}&\alpha_{1,3}^{3}&0&0&0&0&0&0\\
\partial_3^3&0&0&0&0&0&0&\alpha_{1,1}^{3}&\alpha_{1,2}^{3}&0&0&0&0 \\
0&\partial_3^{13}&0&0&0&0&0&0&0&\alpha_{3,3}&\alpha_{3,4}&0&0\\
0&0&\partial_2^{13}&0&0&0&0&0&0&0&0&\alpha_{1,2}^{13}&0\\
0&0&0&\partial_1^{12}&0&0&0&0&0&0&0&0&0\\
0&0&0&0&\partial_2^{12}&0&0&0&0&0&0&0&\alpha_{2,2}^{3}\\
0&0&0&0&0&\partial_1^{9}&0&0&0&0&0&0&0\\
0&0&0&0&0&0&0&\partial_1^{12}&0&0&0&0&0\\
0&0&0&0&0&0&0&0&0&\partial_1^{41}&0&0&0\\
 \end{array}\right).
\end{align*}
Finally, we write the differentials of \eqref{res ex} as matrices in terms of standard $R$-bases, using the following:
\begin{equation*}
\partial_1=
\begin{pmatrix}
x&y&z    
\end{pmatrix}, \qquad 
\partial_2=
\begin{pmatrix}
-y&-z& 0\\
 x& 0& -z\\
 0& x&  y
\end{pmatrix}, \qquad
\partial_3=
\begin{pmatrix}
z\\
-y\\
x
\end{pmatrix}
\end{equation*}
and 
\begin{align*}
z^1_1&=\begin{pmatrix}x\\0\\0\end{pmatrix}, &z^1_2&=\begin{pmatrix}0\\y\\0\end{pmatrix}, &z^1_3&=\begin{pmatrix}0\\0\\z\end{pmatrix}, & z^1_4&= \begin{pmatrix}yz\\0\\0\end{pmatrix}, \\[0.05in]
z^1_1\wedge z^1_2&=\begin{pmatrix}xy\\0\\0\end{pmatrix},&  z^1_2\wedge z^1_3&=\begin{pmatrix}0\\0\\yz\end{pmatrix}, & z^1_1\wedge z^1_3&=\begin{pmatrix}0\\xz\\0\end{pmatrix}, & z^2_1&=\begin{pmatrix}yz\\0\\0\end{pmatrix}, & z^2_2&=\begin{pmatrix}xz\\0\\0\end{pmatrix}, &
  z^2_3&=\begin{pmatrix}0\\yz\\0\end{pmatrix}, && \\[0.05in] 
  z^3_1&=(yz), & z^3_2&=(xz), & z^3_3&=(xy). &&
\end{align*}

\subsection*{Acknowledgment}
We thank the referee for a careful reading of the paper and for their suggestions that improved the exposition of the paper. We thank Sasha Pevzner for suggesting  some arguments in the proof of \Cref{cor: Poi}. Nguyen was partially supported by the NSF grant DMS-2201146. 

\subsection*{Disclaimer} The views expressed in this article are those of the author(s) and do not reflect the official policy or position of the U.S. Naval Academy, Department of the Navy, the Department of Defense, or the U.S. Government.

%%%%%%%%%%%%%%%%%%%%%%%%%%%%%%%%%


\begin{thebibliography}{9}
\bibitem{Av74} L.~L.~Avramov, {\it On the Hopf algebra of a local ring}, Izv. Akad. Nauk SSSR Ser. Mat., \textbf{38} (1974), pp.~253--277.

\bibitem{Av2} L.~L.~Avramov, {\it A cohomological study of local rings of embedding codepth $3$}, J.~Pure Appl.~Algebra, \textbf{216} (2012), no.~11, pp.~2489--2506.

\bibitem{AKM} L.~L.~Avramov, A.~R.~Kustin, and M.~Miller, {\it Poincar\'{e} series of modules over local rings of small embedding codepth or small linking number}, J.~Algebra, {\bf 118} (1988), no.~1, pp.~162--204.

\bibitem{BCLP} B.~Briggs, J.~C.~Cameron, J.~C.~Letz, and J.~Pollitz, {\it Koszul homomorphisms and universal resolutions in local algebra}, Forum of Mathematics, Sigma., \textbf{13} (2025), e63, \url{doi:10.1017/fms.2025.21}.

\bibitem{CVW} L. W. Christensen, O. Veliche, and J. Weyman. {\it Three takes on almost complete intersection ideals of grade 3}, In Peeva, I. (eds) Commutative Algebra, pp.~219–281, Springer, Cham., 2021, \url{https://doi.org/10.1007/978-3-030-89694-2_7}.

\bibitem{BH} W.~Bruns and J.~ Herzog, {\it Cohen-{M}acaulay rings}, {Cambridge University Press}, vol.~{\bf 39}, {Cambridge studies in advanced mathematics}, 1993.

\bibitem{G} E.~S.~Golod, {\it On the homology of some local rings}, Dokl. Akad. Nauk SSSR, \textbf{144} (1962), pp.~479--482 (in Russian); English translation: Soviet Math. Dokl. \textbf{3} (1962), pp.~745--748.

\bibitem{M2} D.~R.~Grayson and M.~Stillman, {\it Macaulay2, a software system for research in algebraic geometry}, Available at \url{http://www.math.uiuc.edu/Macaulay2/}.

\bibitem{NV} V.~C.~Nguyen and O.~Veliche, {\it A truncated minimal free resolution of the residue field}, In: Miller, C., Striuli, J., Witt, E.E. (eds) Women in Commutative Algebra, {\it Proceedings of the 2019 WICA Workshop}, Association for Women in Mathematics Series, vol.~{\bf 29}, pp.~399--437, Springer, 2022. 

\bibitem{NV2} V.~C.~Nguyen and O.~Veliche, {\it Iterated mapping cones on the Koszul complex and their application to complete intersection rings}, Journal of Algebra and Its Applications, {\bf 24} (2025), No.~10, 2550234, \url{https://doi.org/10.1142/S0219498825502342}.

\bibitem{T} J.~Tate, {\it Homology of Noetherian rings and local rings}, Illinois J.~Math., {\bf 1} (1957), pp.~14--27.

\bibitem{W} J.~Weyman, {\it On the structure of free resolutions of length 3}, J.~Algebra {\bf 126} (1989), no.~1, pp.~1--33.

\end{thebibliography}
\end{document}